\documentclass[12pt]{amsart}

\makeatletter
\def\@citecolor{blue}
\def\@urlcolor{blue}
\def\@linkcolor{blue}
\makeatother
\usepackage{tikz}
\usetikzlibrary{calc,intersections,through,backgrounds,patterns}
\pgfdeclarepatternformonly{my crosshatch dots}{\pgfqpoint{-1pt}{-1pt}}{\pgfqpoint{5pt}{5pt}}{\pgfqpoint{6pt}{6pt}}%
{
    \pgfpathcircle{\pgfqpoint{0pt}{0pt}}{.5pt}
    \pgfpathcircle{\pgfqpoint{3pt}{3pt}}{.5pt}
    \pgfusepath{fill}
}

\usepackage{enumerate}
\usepackage[all]{xy}  
\usepackage{graphicx}
\usepackage[headings]{fullpage}
\usepackage{amsmath,amsthm,amssymb,amsfonts,latexsym,mathrsfs}
\usepackage{color}
\usepackage{ifthen}
\usepackage{wrapfig}
\usepackage[colorlinks,pagebackref=true,pdftex]{hyperref}
\usepackage[normalem]{ulem}
\usepackage{comment}

\newcommand{\Ht}{{\rm ht}}

\newcommand{\ann}{{\rm ann}}
\makeatletter

\@addtoreset{equation}{section}
\def\theequation{\thesection.\@arabic \c@equation}
\def\@citecolor{blue}
\def\@urlcolor{blue}
\def\@linkcolor{blue}
\def\theenumi{\@roman\c@enumi}

\makeatother

\theoremstyle{plain}
\newtheorem*{claim*}{Claim}

\theoremstyle{definition}
\newtheorem{theorem}{Theorem}[section]
\newtheorem{lemma}[theorem]{Lemma}
\newtheorem{corollary}[theorem]{Corollary}
\newtheorem{proposition}[theorem]{Proposition}
\newtheorem{remark}[theorem]{Remark}

\newtheorem{remarks}[theorem]{Remarks}

\newtheorem{example}[theorem]{Example}

\newtheorem{definition}[theorem]{Definition}

\newtheorem{conjecture}[theorem]{Conjecture}

\newtheorem{question}[theorem]{Question}

\newcommand{\m}{\mathfrak{m}}

\def\NZQ{\mathbb}               
\def\NN{{\NZQ N}}
\def\ZZ{{\NZQ Z}}
\def\PP{{\NZQ P}}
\def\CC{{\NZQ C}}
\def\frk{\mathfrak}               
\def\aa{{\frk a}}

 \newcommand{\ul}[1]{
\underline{#1}}
\newcommand{\ov}[1]{\overline{{#1}}}

\def\opn#1#2{\def#1{\operatorname{#2}}} 
\opn\chara{char}
\opn\length{\ell}
\opn\projdim{proj\,dim}
\opn\depth{depth}
\opn\reg{reg}
\opn\lreg{lreg}
\opn\sat{^{sat}}
\opn\lex{^{lex}}
\opn\Ker{Ker}
\opn\Coker{Coker}
\opn\Im{Im}
\opn\Hom{Hom}
\opn\Tor{Tor}
\opn\Ext{Ext}
\opn\End{End}
\opn\Aut{Aut}
\opn\id{id}
\opn\GL{GL}
\renewcommand{\leq}{\leqslant}
\renewcommand{\geq}{\geqslant}

\let\:=\colon
\let\ov\overline
\let\lra\longrightarrow
\let\un\underline
\opn\Gin{Gin}

\opn\Hilb{Hilb}
\opn\ini{in}
\opn\End{end}

\newcommand{\giulio}[1]{{\color{blue} \sf #1}}

\newcommand{\EGHH}[2]{{\rm EGH}_{#1}(#2)}
\newcommand{\EGHHH}[1]{{\rm EGH}_{#1}}
\newcommand{\LPP}[2]{{\rm LPP}^{#1}(#2)}

\title{The Eisenbud-Green-Harris Conjecture}

\author{Giulio Caviglia}
\address{Giu\-lio Ca\-vi\-glia - Department of Mathematics -  Purdue University - 150 N. University Street, West Lafayette - 
  IN 47907-2067 - USA} 
\email{gcavigli@math.purdue.edu}
\author{Alessandro De Stefani}
\address{Alessandro De Stefani - Dipartimento di Matematica, Universit{\`a} di Genova, Via Dodecaneso 35, 16146 Genova, Italy}
\email{destefani@dima.unige.it}
\author{Enrico Sbarra}
\address{Enrico Sbarra - Dipartimento di Matematica - Universit\`a degli Studi di Pisa -Largo Bruno Pontecorvo 5 - 56127 Pisa - Italy}
\email{enrico.sbarra@unipi.it}

\begin{document}
\begin{abstract}
We survey most of the known results concerning the Eisenbud-Green-Harris Conjecture. Our presentation includes new proofs of several theorems, as well as a unified treatment of many results which are otherwise scattered in the literature. We include a final section with some applications, and examples.
\end{abstract}

\maketitle
\section{An introduction to the conjecture} \label{Section statements}

A very important problem in Commutative Algebra is the study of the growth of the Hilbert function of an ideal in a given degree 
\emph{if one knows more than one step of {\rm [its]} history}, 
cit. Mark Green \cite{Green}. A classical theorem, due to Macaulay \cite{Ma}, answers this question by providing an estimate on the Hilbert function in a given degree just by knowing its value in the previous one. This result is very useful, but it is far from being optimal. For instance, there is no way of taking into account any additional information about the ideal. The Eisenbud-Green-Harris, henceforth EGH, Conjecture was first raised in \cite{EiGrHa1,EiGrHa2}, and precisely addresses this matter. By effectively using the additional data that the given ideal contains a regular sequence, it predicts for instance more accurate growth bounds.

We will now introduce some notation and terminology in order to state the EGH Conjecture. Throughout this article, $A = \bigoplus_{d \geq 0} A_d$ will denote a standard graded polynomial ring $K[x_1,\ldots,x_n]$ over a field $K$, and $\m=(x_1,\ldots,x_n)$ its homogeneous maximal ideal. 
We consider $A$ equipped with the lexicographic 
order $\geq$ induced by $x_1>x_2>\ldots > x_n$. Given polynomials 
$g_1,\ldots,g_s \in A$, we will denote by $\langle g_1,\ldots,g_s \rangle$ the $K$-vector space generated by such elements to distinguish it from the ideal that they generate, which we denote by $(g_1,\ldots,g_s)$. We denote the Hilbert function of a graded module $M$ and its value in $d$ by $H(M)$ and $H(M; d)$, respectively. On the set of Hilbert functions 
we consider the partial order given by point-wise inequality. 
Recall that a $K$-vector space $V \subseteq A_d$ is called {\it lex-segment} if there exists a monomial $v \in V$ such that $V= \langle u \in A_d \mid u \text{ monomial}, u \geq v \rangle$.

The classical Macaulay Theorem states that, given any homogeneous ideal $I$, if one lets $L_d \subseteq A_d$ be the lex-segment of dimension equal to $H(I;d)$, then ${\rm Lex}(I)=\bigoplus_{d \geq 0} L_d$ is an ideal, that we call {\it lex-ideal}. In order to take into account that $I$ contains a regular sequence, we will introduce the so-called lex-plus-powers ideals.

Given an integer $0 < r\leq n$, we let $\underline{a}=(a_1,\ldots,a_r)$ denote an ordered sequence of integers $0<a_1\leq\ldots\leq a_r$, and we call it a \emph{degree sequence}. We call the ideal $\aa = (x_1^{a_1},\ldots,x_r^{a_r}) \subseteq A$ the {\em pure-powers ideal of degree $\underline{a}$}. 
With any homogeneous ideal $I\subseteq A$ which contains an ideal ${\bf f}$ generated by a regular sequence $f_1,\ldots,f_r$, of degree $\underline{a} = (a_1,\ldots,a_r)$, we associate the $K$-vector space 
\[
\LPP{\ul{a}}{I}=\bigoplus_{d\geq 0} \langle L_d+\aa_d\rangle,
\]
where 
$L_d\subseteq A_d$ is the largest, 
hence unique, lex-segment which satisfies $H(I;d)=\dim_K\langle L_d+\aa_d \rangle$. As Macaulay Theorem proves that ${\rm Lex}(I)$ is an ideal, the EGH Conjecture predicts that $\LPP{\ul{a}}{I}$ is an ideal, which we call {\em the lex-plus-powers ideal associated with $I$ with respect to the degree sequence $\underline{a}$}. 

\begin{conjecture}[EGH] \label{Conj EGH} Let $I \subseteq A$ be a homogeneous ideal that contains a homogeneous ideal ${\bf f}$ generated by regular sequence of degree $\ul{a}$. Then $\LPP{\ul{a}}{I}$ is an ideal.
\end{conjecture}

Observe that the EGH Conjecture is a generalization of Macaulay Theorem, which corresponds to the case ${\bf f} = (f_1)$ with respect to any $0 \ne f_1 \in I$ of degree $a_1$. Just like lexicographic ideals, lex-plus-powers ideals enjoy several properties of extremality. For example, assuming that the EGH Conjecture is true in general, then one can show that the growth of $\LPP{\ul{a}}{I}$ in each degree is smaller than that of $I$. That is, $H(\m \LPP{\ul{a}}{I}) \leq H(\m I)$, see Lemma \ref{Lemma lpp growth bis}. This immediately translates into an inequality $\beta_{0j}(\LPP{\ul{a}}{I}) \geq \beta_{0j}(I)$ between minimal number of generators in each degree $j$. We point out that the more refined version of such inequality, i.e., 
\[
\beta_{ij}(\LPP{\ul{a}}{I}) \geq \beta_{ij}(I) \, \, \, \, \, \quad \text{ for all } i,j,
\]
is currently unknown in general, and goes under the name of LPP-Conjecture, see for instance \cite{Fr,Ri,FrRi,MePeSt,MeMu,CaSa}.

In the following, it will be useful to have several formulations of the EGH Conjecture, which we will use interchangeably at our convenience. 

An equivalent way of approaching the conjecture is degree by degree: given a sequence $\ul{a}$, for a non-negative integer $d$ we say that a homogeneous ideal $I \subseteq A=K[x_1,\ldots,x_n]$ \emph{satisfies $\EGHH{\ul{a}}{d}$} if there exists an $\ul{a}$-lpp ideal $J$ such that $\dim_K(J_d) = \dim_K(I_d)$ and $\dim_K(J_{d+1}) \leq \dim_K(I_{d+1})$. We say that $I$ \emph{satisfies $\EGHHH{\ul{a}}$} if it satisfies $\EGHH{\ul{a}}{d}$ for all non-negative integers $d$. One can readily verify that Conjecture \ref{Conj EGH} holds true if and only if, for every degree sequence $\ul{a}$, every homogeneous ideal containing a regular sequence of degree $\ul{a}$ satisfies $\EGHHH{\ul{a}}$, see \cite{CaMa}.



We conclude this introductory section by recalling a weaker version of the EGH Conjecture, raised in \cite{EiGrHa2}. Let $\ul{a} = (a_1,\ldots,a_n)$ be a degree sequence, and $D$ be an integer such that $a_1 \leq D \leq \sum_{i=1}^n (a_i-1)$. Let $b$ the unique integer such that $\sum_{i=1}^b (a_i-1) \leq D < \sum_{i=1}^{b+1} (a_i-1)$, and set $\delta = \sum_{i=1}^{b+1} (a_i-1)-D+1$ if $b<n$, and $\delta=1$ otherwise.

\begin{conjecture}[Cayley-Bacharach] \label{Conj CB}Let ${\bf f}\subseteq A=K[x_1,\ldots,x_n]$ be an ideal generated by a regular sequence of degree $\ul{a} = (a_1,\ldots,a_n)$, and $g \notin {\bf f}$ be a homogeneous element of degree $D \geq a_1$. Let $I={\bf f}+(g)$, and $e$ be the multiplicity of $A/I$. Then
\[
e\leq \prod_{i=1}^n a_i  - \delta \prod_{i=b+1}^n a_i.
\] 
\end{conjecture}

Conjecture \ref{Conj CB} has been studied by several researchers, from very different points of view; for instance, see \cite{GeKrRo,GeKr,CaDeS2,HoLaUl}. The validity of the EGH Conjecture in the case $r=n$ for almost complete intersections would imply Conjecture \ref{Conj CB}.  
For an explicit instance of this, see Example \ref{Ex CB}. 

\medskip

This survey paper is structured as follows: in Section \ref{Section CL} we treat the case when the given ideal already contains a pure-powers ideal, presenting a new proof of the Clements-Lindstr{\"o}m Theorem. Section \ref{Section Artinian + Linkage} is very brief, and collects some statements from the theory of linkage, together with a result which yields a reduction to the Artinian case. In Section \ref{Section EGH} we present proofs of several cases of the conjecture, previously known in the literature. Finally, in Section \ref{Section Applications Examples} we collect some applications of the techniques and the results illustrated before, together with several examples.

\section{Monomial regular sequences and the Clements-Lindstr\"om Theorem} \label{Section CL}


The goal of this section is to prove the Clements-Lindstr{\"o}m Theorem \cite{ClLi}, a more general version of the Kruskal-Katona Theorem \cite{Kr,Ka}. The proof presented here relies on recovering a strong hyperplane restriction theorem 
for strongly-stable-plus-powers and lpp ideals due to Gasharov \cite{Ga2,Ga}, see also \cite[Theorem 2.2]{CaSb}. Our strategy uses the techniques of \cite{CaKu1}, and is different from the standard one available in the literature \cite{ClLi,MePe1,MePe2}.

Recall that a monomial 
ideal $J \subseteq A=K[x_1,\ldots,x_n]$ is called 
{\it strongly stable} if for every monomial $u \in J$ and any variable $x_i$ which divides $u$, one has that $x_i^{-1}x_ju \in J$ for all $1 \leq j \leq i$. The ideal $J$ is said to be {\em $\underline a$-strongly-stable-plus-powers}, $\underline a$-{\rm spp} or, simply, spp for short,  if there exist a strongly stable ideal $S$ and a pure power ideal $\aa$ of degree $\underline{a}$ such that $J=S+\aa$. Clearly, $\ul{a}$-lpp ideals are $\ul{a}$-spp.

\begin{theorem}\label{cleli} 
Let $I \subseteq A$ be a homogeneous ideal that contains a pure-powers ideal $\aa$ of degree ${\underline a}$. Then
\begin{enumerate}
\item \label{cleli1} $\LPP{\ul{a}}{I}$ is an ideal.
\item \label{cleli2} If $I$ is $\ul{a}$-spp, then $H(I+(x_n^i)) \geq H(\LPP{\ul{a}}{I} + (x_n^i))$ for all $i>0$. 
\end{enumerate}
\end{theorem}

\noindent
We first prove Theorem \ref{cleli} (\ref{cleli1}) for $n=r=2$. Since strongly stable ideals in two variables are lex-ideals, $\ul{a}$-spp ideals are automatically $\ul{a}$-lpp in this case. 

We start by 
recalling  a few  properties of monomial ideals, which are special cases of more general results derived from linkage theory, that we will discuss in Section \ref{Section Artinian + Linkage}.

Let $I$ be a monomial ideal that  contains  $\aa=(x_1^{a_1}$, $x_2^{a_2})$. 
When we view $I$ as a $K[x_1]$-module, we have a decomposition
\begin{equation}\label{linkato1}
  I= \bigoplus_{i \geq 0} x_1^{d_i} K[x_1] \cdot x_2^{i};
\end{equation}
observe that, since $I$ is an ideal, one has $d_i\geq d_{i+1}$ for all $i$. Also observe that $I$ is spp if and only if $d_{i+1}+1\geq d_i$ for all $i$. Define the link $I^{\ell}=I^\ell_{\aa}$ of $I$ with respect to the ideal $\aa$ to be the ideal $I^\ell=(\aa:_AI)$. Notice that $I^\ell=(x_1^{a_1-d_0},x_2^{a_2})\cap (x_1^{a_1-d_1},x_2^{a_2-1})\cap\cdots\cap (x_1^{a_1-d_{a_2-1}},x_2)$ is an ideal generated by the monomials 
$x_1^{a_1-d_i}x_2^{a_2-1-i}$,\, $i=0,\ldots,a_2-1$, and that as a $K[x_1]$-module can be written as 
\begin{equation}\label{linkato2}
  I^\ell=\left(\bigoplus_{i=0}^{a_2-1} x_1^{a_1-d_{a_2-1-i}}K[x_1]\cdot x_2^i\right)\oplus \left(\bigoplus_{i\geq a_2} K[x_1] \cdot x_2^i\right).
\end{equation}

\begin{remark}\label{linklexlex} 
\begin{enumerate}[(1)]
\item \label{linkmon1} It is immediate from \eqref{linkato2} that $(I^\ell)^\ell=I$. 
\item \label{linkmon2} The Hilbert function of $I^\ell$ is determined by that of $I$. More precisely, if we let $R=A/\aa$ and $s=a_1+a_2-2$, then $H(R;d)=H(R/IR;d)+H(R/I^\ell R;s-d)$.
\item \label{linkmon3} The link of an $\ul{a}$-lpp ideal is again an ${\ul a}$-lpp ideal. 
Thus, we may as well prove that $I^\ell$ is $\ul{a}$-spp if $I$ is $\ul{a}$-spp. To this end, consider the decomposition of $I$ as in \eqref{linkato1}. Given any monomial 
$x_1^{b_1}x_2^{b_2}\in I^\ell$ with $1 \leq b_2 <a_2$, one just needs to show that $x_1^{b_1+1}x_2^{b_2-1}\in I^\ell$. By \eqref{linkato2}, it is enough to verify that $a_1-d_i+1\geq a_1-d_{i+1}$ for all $i$, which is equivalent to $d_{i+1}+1\geq d_i$ for all $i$. Finally, this is true for all $i$,  because $I$ is spp by assumption.
\end{enumerate}
\end{remark}


We are now ready to prove the case $n=2$ of Theorem \ref{cleli} (i).
\begin{proposition}\label{nequal2}
Let $\ul{a} = (a_1,a_2)$, and $I \subseteq A= K[x_1,x_2]$ be a homogeneous 
ideal that contains $\aa = (x_1^{a_1},x_2^{a_2})$. Then $\LPP{\ul{a}}{I}$ is an ideal. 
\end{proposition} 

After taking any initial ideal, without loss of generality we may assume that $I$ is monomial. In fact, this operation preserves its Hilbert function, 
and the initial ideal still contains $\aa$. Next, we give three different proofs of the above 
proposition.

In the first one, we make use of linkage.
\begin{proof}[{\bf Proof 1}]
We need to show that the $K$-vector space $\LPP{\ul{a}}{I}=\bigoplus_{j\geq 0} \langle L_j+\aa_j\rangle $ is indeed an ideal, and we do so by proving that $\LPP{\ul{a}}{I}$ agrees with an ideal $J$ 
for all degrees $i\leq a_2-1$ and it agrees with an ideal $J'$ for all degrees $i\geq a_2-1$.
By Macaulay Theorem, there is a lex-ideal $L$ with the same Hilbert function as $I$. Consider the $\ul{a}$-lpp ideal $J=L+\aa$. By construction, for all $j=1,\ldots,a_2-1$ one has 
$H(J;j)=H(L;j)=H(I;j)$.
  
Now we construct the ideal $J'$ as follows. First consider the link $I^\ell=(\aa:_AI)$. Since $I^\ell\supseteq \aa$, again by Macaulay Theorem there exists a lexicographic ideal $L'$ with the same Hilbert function as $I^{\ell}$. Thus, the $\ul{a}$-lpp ideal 
$L'+\aa$ has the same Hilbert function as $I^\ell$ in degrees $j=0,\ldots,a_1-1$. We now let $J' = (L'+\aa)^\ell$. By Remark \ref{linklexlex} (\ref{linkmon3}), $J'$ is an lpp ideal and, by Remark \ref{linklexlex} (\ref{linkmon2}) its Hilbert function in degrees $j\geq a_2-1$ coincides with that of $I$. Therefore $J'$ has the desired properties, and the proof is complete.
\end{proof}

In the second proof we use techniques borrowed from \cite[Section 3]{MeMu}, see also \cite[Section 4]{CaKu1}.

\begin{proof}[{\bf Proof 2}] 
The Hilbert function of a monomial ideal is independent of the base field, thus without loss of generality we may assume that $K= \CC$. 
It suffices to construct an $\underline{a}$-spp ideal with the same Hilbert function as $I$. Let $\xi_1,\dots,\xi_{a_2}$ the $a_2$-roots of unity over $\CC$, and observe that 
$x_2^{a_2}-x_1^{a_2}= (x_2-\xi_1 x_1)(x_2-\xi_2x_1) \cdots (x_2-\xi_{a_2}x_1)\in I$. We consider the 
distraction  
$\mathcal D$ given by a family of linear forms $\{l_i\}_{i \geq 1}$ defined as $l_i=x_2-\xi_ix_1$, for $i=1,\ldots,a_2$, and $l_i=x_2$ for all $i> a_2$; see \cite{BiCoRo} for the theory of distractions. Given a decomposition of  $I^{(0)}=I= \bigoplus_{i \geq 0} I_{[i]} x_2^{i}$, we let $J^{(0)}$ be the distracted ideal
\[
J^{(0)}=J=\,\bigoplus_{i\geq 0} I_{[i]}\prod\limits_{j=1}^{i} l_j\,=\,\bigoplus_{i=0}^{a_2}  I_{[i]} \prod\limits_{j=1}^{i} l_j\oplus\bigoplus_{i\geq a_2} K[x_1] \cdot x_2^i,
\]
which shares with $I$ the same Hilbert function, and the same Betti numbers as well. 
Observe that the last equality is due to the fact that both $x_1^{a_2}$ and $x_2^{a_2}-x_1^{a_2}$ are in $J$, and therefore $x_2^{a_2} \in J$. We let $I^{(1)}$ be ${\rm in}_{>}(J^{(0)})$, where $>$ is any monomial order such that $x_1>x_2$, and $J^{(1)}$ be the ideal obtained by distracting $I^{(1)}$ with $\mathcal{D}$. We construct in this way  a sequence $I^{(0)}$, $I^{(1)}$,\ldots,$I^{(h)}$ of ideals with the same Hilbert function, each of which contains $\aa$; we finally want to show  that this sequence eventually stabilizes at an ideal, 
we call it $L$,  which is $\underline{a}$-spp. To this end, observe that for all integers $p \geq 0$ we have 
   \begin{equation}\label{unosei}\begin{split}
     H\left( I^{(h)}_{[0]}\oplus I^{(h)}_{[1]}x_2\oplus\cdots\oplus
     I^{(h)}_{[p]}x_2^p\right) &= H\left (\ini_>( I^{(h)}_{[0]}\oplus
     I^{(h)}_{[1]}l_1\oplus\cdots\oplus
     I^{(h)}_{[p]}\prod_{j=1}^{p}l_j)\right )\\ &
     \leq H\left ( I^{(h+1)}_{[0]}\oplus
     I^{(h+1)}_{[1]}x_2\oplus\cdots\oplus
     I^{(h+1)}_{[p]}x_2^p\right).
     \end{split}
   \end{equation}
   In the above, we consider three modules whose Hilbert functions are computed as homogenous $K[x_1]$-submodules of the graded $K[x_1]$-module $A = K[x_1,x_2]$, where $x_2^d$ has degree $d$.
   Notice that the inequality in \eqref{unosei} is due to the inclusion of the second module in the third one. Observe that
   $I^{(0)}_{[0]}\subseteq I^{(1)}_{[0]} \subseteq \ldots$ 
is an ascending chain of ideals that will eventually stabilize, say at $I_{[0]}^{(h_0)}$. 	Inductively, assume that for all $i=0,\ldots,p-1$  the ideals in $I^{(h_{i-1})}_{[i]}\subseteq I^{(h_{i-1}+1)}_{[i]}\subseteq\ldots$ form a chain that stabilizes, say at $h_{i}$. 
The inclusion of the second into the third module of \eqref{unosei}, for any $h>\max\{h_0,\dots,h_{p-1}\}$, yields that  $I^{(h)}_{[p]} \subseteq I^{(h+1)}_{[p]}$. Thus, for  $h\geq h_{p-1}$  we have again a chain of ideals which will stabilize, say at $h_p$. Repeat this process for all $p \leq a_2-1$, so that for all $h \geq h'= \max\{h_1,\ldots,h_{a_2-1}\}$ we have $I^{(h)} = I^{(h+1)}$. Let $L = I^{(h')}$. 
Keeping in mind how $L$ has been constructed, apply \eqref{unosei} to $L$ to obtain, for all $p\geq 0$

 \begin{equation}\label{unoelle}\begin{split}
     L_{[0]}\oplus L_{[1]}x_2\oplus\cdots\oplus
     L_{[p]}x_2^p &= \ini_>\left( L_{[0]}\oplus
     L_{[1]}l_1\oplus\cdots\oplus
     L_{[p]}\prod_{j=1}^{p}l_j\right)\\
     &=L_{[0]}\oplus
     L_{[1]}l_1\oplus\cdots\oplus
     L_{[p]}\prod_{j=1}^{p}l_j,
   \end{split}
 \end{equation}
 where the second equality can be verified by induction on $p$, using the first equality and the fact that the least monomial with respect to $>$ in the support of $\prod_{j=1}^{p}l_j$ is $x_2^p$.

Next, we prove that $L$ is $\underline{a}$-spp. By construction $L\supseteq \aa$, since each $I^{(i)}$ and $J^{(i)}$ does; thus, we have to show that $x_1L_{[p]}\subseteq L_{[p-1]}$ that for all $0<p\leq a_2-1$. Again by induction on $p$, by \eqref{unoelle} we have
 $L_{[0]}\oplus L_{[1]}x_2=L_{[0]}\oplus L_{[1]}(x_2-x_1)$, which implies $x_1L_{[1]}\subseteq L_{[0]}$. Moreover, by induction and again by \eqref{unoelle},
 $L_{[0]}\oplus L_{[1]}x_2\oplus\cdots\oplus
 L_{[p]}x_2^p=L_{[0]}\oplus L_{[1]}x_2\oplus\cdots\oplus
 L_{[p-1]}x_2^{p-1}\oplus L_{[p]}\prod_{j=1}^{p}l_j$. Since $l_j=x_2-\xi_j x_1$ with $j=1,\ldots,p$, we have that  $\prod_{j=1}^{p}l_j$ has a full support, i.e., its support contains all of the monomials of degree $p$. In particular  it contains $x_1x_2^{p-1}$. It follows that
 $x_1L_{[p]}\subseteq L_{[p-1]}$, as desired.
\end{proof}

The third proof relies on an application of Gotzmann Persistence Theorem \cite{Got,Gre}. 
\begin{proof}[{\bf Proof 3}] 
Let $\LPP{\ul{a}}{I} = L+\aa$, where each $L_d$ is the largest lex-segment such that $\dim_K(L_d+\aa_d) = H(I;d)$. In order to show that $\LPP{\ul{a}}{I}$ is an ideal we have to show that, for every integer $d \geq 0$, 
we have $H(A/(\m L+\aa); d+1)\geq H(A/\LPP{\ul{a}}{I}; d+1)$. For this, without loss of generality we can assume that $(\LPP{\ul{a}}{I})_j = \aa_j$ for all $j<d$. Let $k=\dim_K \widetilde{L_d}$, where $\widetilde{L_d}$ is the image in $A/\aa$ of the $K$-vector space $L_d+\aa_d$. If $k=0$ there is nothing to prove. Let us assume $k>0$, and study the following three cases separately: $d< a_2-1$, $d=a_2-1$, and $d \geq a_2$. If $d< a_2-1$, then 
$(\LPP{\ul{a}}{I})_d = L_d$, $(\LPP{\ul{a}}{I})_{d+1} = L_{d+1}$, and the conclusion follows from Macaulay Theorem. 

Now assume $d=a_2-1$. 
If $L_{d+1} = A_{d+1}$, then there is nothing to show, so assume that $L_{d+1} \subsetneq A_{d+1}$. If $x_2^{d+1}$ is a minimal generator of $I$, then $H(A/(\m L_d);d+1) \geq H(A/\m I;d+1) \geq H(A/I;d+1) +1$. Since $H(A/I;d+1) = H(A/L;d+1) - 1$, it follows that $H(A/\m L;d+1) \geq H(A/L;d+1)$, and therefore $\m_1 L_d \subseteq L_{d+1}$. A fortiori, we have that $\m_1 (\LPP{\ul{a}}{I})_d \subseteq (\LPP{\ul{a}}{I})_{d+1}$, and the proof of this case is complete. If $x_2^{d+1}$ is not a minimal generator of $I$, then $\dim(A/J) = 0$, where $J=I_{\leq d}$. In particular, $H(A/I;j) \leq H(A/(x_1^d,x_2^d);d) = d$. By Macaulay Theorem we have that $H(A/I;d+1) \leq d$. If equality holds, then $I$ has no minimal generators in degree $d+1$, and thus $H(A/J;d+1) = d$ as well. By Gotzmann Persistence Theorem applied to $J$, we have that $H(A/H;j) = d$ for all $j \geq d$, which contradicts the fact that $\dim(A/J)=0$.
  
Finally, if $d \geq a_2$, we first observe that once again $k=H(A/I;d) \leq d$, and that $H(A/(\m L_d)+\aa;d+1) = k-1$. Since $k = H(A/I;d)$, to conclude the proof it suffices to show that $k > H(A/I;d+1)$, since the latter is equal to $H(A/\LPP{\ul{a}}{I};d+1)$. It follows from Macaulay Theorem $H(A/I;d+1) \leq k=H(A/I;d)$, since we have already observed that $k \leq d$. If equality holds, then by Gotzmann Persistence Theorem applied to the ideal $J=I_{\leq d}$ we would have that $H(A/J;j) = H(A/J;d) = k>0$ for all $j\geq d$. In particular, this would imply that $\dim(A/J)>0$, in contrast with the fact that $J$ contains $(x_1^{a_1},x_2^{a_2})$, and hence it is Artinian.
\end{proof}

\begin{remark} \label{Remark three proofs} (1) Observe that Proof 1 can be adapted to any regular sequence of degree $\ul{a}=(a_1,a_2)$ using properties of linkage analogous to those of Remark \ref{linklexlex}, see Theorem \ref{Thm EGH large degrees}. 

\noindent (2) It is easy to see that, in Proof 2, we can also keep track of Betti numbers and prove, in characteristic zero, that they cannot decrease when passing to the lex-plus-powers ideal. 

\noindent(3) In Proof 3 we do not actually use the fact that the regular sequence is monomial. In fact, the same argument can be used to prove that any ideal which contains a regular sequence of degree $\ul{a}=(a_1,a_2)$ satisfies $\EGHHH{\ul{a}}$.
\end{remark}

We now move our attention from the case $n=2$ to the general one.

\begin{proposition} \label{Proposition stabilize}
Under the same assumptions of Theorem \ref{cleli}, there exists an  $\underline a$-{\rm spp} ideal with  the same Hilbert function as that of $I$. 
\end{proposition}
\begin{proof}
We define a total order on the set $\mathcal S$ 
of monomial ideals with the same Hilbert function as $I$, and which contain the pure-powers ideal $\aa=(x_1^{a_1},\ldots,x_r^{a_r})$. First, given any $J \in \mathcal S$, we order the set of its monomials $\{m_i\}$ from lower to higher degrees, and monomials of the same degree lexicographically. Now, given a second ideal $J' \in \mathcal S$ and the set of its monomials $\{m_i'\}$, we set $J > J'$ if and only if there exists $i$ such that $m_j = m_j'$ for all $j \leq i$ and $m_{i+1} > m_{i+1}'$. Observe that, since $J$ and $J'$ have the same Hilbert function, we are forced to have $\deg m_j = \deg m_j'$ for all $j$. Let $P$ be the maximal element of $\mathcal S$; we claim that $P$ is $\ul{a}$-spp. Assume by contradiction that there exists a monomial $m \in P \smallsetminus \aa$ such that $x_i$ divides $m$ and $x_i^{-1}x_j m \notin P$ for some $j < i$. Write $P = \bigoplus_q P_q \cdot q$, where each $q \in K[x_1,\ldots,\widehat{x_j},\ldots,\widehat{x_i},\ldots,x_n]$ is a monomial, and $P_q \subseteq K[x_j,x_i]$ is an ideal. Notice that each $P_q$ contains $(x_j^{a_j},x_i^{a_i})K[x_j,x_i]$ since $P \in \mathcal S$, and that $P_q \subseteq P_{q'}$ whenever $q$ divides $q'$ since $P$ is an ideal. By Proposition \ref{nequal2}, for every $q$ there exists an $(a_j,a_i)$-spp ideal $Q_q \subseteq K[x_j,x_i]$ with the same Hilbert function as $P_q$. 

Let now $Q=\bigoplus_q Q_q \cdot q$, and observe that $Q \in \mathcal S$. In fact, $Q$ is clearly spanned by monomials, and it contains $\aa$. Moreover, if $q$ divides $q'$ one gets $H(Q_q) = H(P_q) \leq H(P_{q'}) = H(Q_{q'})$. Since $Q_q$ and $Q_{q'}$ are both $(a_j,a_i)$-spp, it follows that $Q_q \subseteq Q_{q'}$, which in turn that $Q$ is an ideal. Since $P$ is not $\ul{a}$-spp, by our choice of the indices $i$ and $j$ there 
 exists $q$ such that $P_q$ is not $(a_j,a_i)$-spp. In particular, it follows that $Q > P$, which contradicts maximality of $P$. 
\end{proof}

\begin{remark} \label{Remark Betti stability}
As in the case of two variables, see Remark \ref{Remark three proofs} (2), 
in the proof of Proposition \ref{Proposition stabilize} one can keep track of how the Betti numbers change in order to prove that, in characteristic zero, the Betti numbers of the $\ul{a}$-spp ideal we obtain cannot decrease. 
This fact is helpful in order to prove the LPP-Conjecture for ideals that contain pure-powers.
\end{remark}


We point out that, in all pre-existing proofs of Clements-Lindstr{\"o}m Theorem \ref{cleli} \cite{ClLi,MePe1,MePe2}, one finds a preliminary reduction step that goes under the name of \emph{compression}. This step consists of assuming that Clements-Lindstr{\"o}m Theorem holds in $n-1$ variables in order to construct an $\ul{a}$-spp ideal $J \subseteq A$ in $n$ variables that, for any $i=1,\ldots,n$, has a decomposition $J = \bigoplus_{j \geq 0} J_{[j]}x_i^j$, where $J_{[j]}$ is $(a_1,\ldots,\widehat{a_i},\ldots,a_r)$-lpp for all $j$. In our proof, this step corresponds to the reduction provided by Proposition \ref{Proposition stabilize}. 
Observe that the above ideal $J$ is not necessarily $\ul{a}$-lpp globally in $n$ variables, as the following example shows. 
\begin{example} \label{Example lpp local not global}
Let $n \geq 4$ and consider the $(2,2)$-spp ideal $I=(x_1^2,x_1x_2,\ldots,x_1x_{n-1},x_2^2,x_2x_3)$ in $A=K[x_1,\ldots,x_n]$; then $I$ is compressed, but not $(2,2)$-lpp, since the monomial $x_1x_n$ is missing from its generators.
\end{example}

We introduce some notation and terminology, which will be used henceforth in this section. 

\noindent Let $A=K[x_1,\ldots,x_n]$, $\m = (x_1,\ldots,x_n)$, $\ul{a} = (a_1,\ldots,a_r)$ be a degree sequence, and $\aa=(x_1^{a_1},\ldots,x_r^{a_r})$ be the corresponding pure-powers ideal. Furthermore, let $\ov{A}=K[x_1,\ldots,x_{n-1}]$, and $\ov{\m} = (x_1,\ldots,x_{n-1})\ov{A}$. If $r<n$, we let $\ov{\ul{a}} = \ul{a}$ and $\ov{\aa}= (x_1^{a_1},\ldots,x_r^{a_{r}})\ov{A}$. Otherwise, if $r=n$, we let $\ov{\ul{a}}  = (a_1,\ldots,a_{n-1})$ and $\ov{\aa} = (x_1^{a_1},\ldots,x_{n-1}^{a_{n-1}}) \ov{A}$.  

\noindent Given a $K$-vector space $V \subseteq A_d$ generated by monomials, we say that $V$ is \emph{$\ul{a}$-lpp} if it is the truncation in degree $d$ of an $\ul{a}$-lpp ideal. Similarly, we say that $V$ is \emph{$\ul{a}$-spp} if it is the truncation in degree $d$ of an $\ul{a}$-spp ideal. Observe that a $K$-vector subspace $V=\bigoplus_{i=0}^d V_{[d-i]}x_n^i \subseteq A_d$ containing $\aa_d$ is $\ul{a}$-spp if and only if $V_{[i]}$ is $\ov{\ul{a}}$-spp for all $i$, and $\ov{\m}_1 V_{[i]} \subseteq V_{[i+1]}$ for all $i \geq \max\{d-a_n+1,0\}$; we will refer to the latter property as \emph{stability}. Moreover, if $V \subseteq A_d$ is $\ul{a}$-lpp, respectively $\ul{a}$-spp, then $\m_1 V+\aa_{d+1}$ is also $\ul{a}$-lpp, respectively $\ul{a}$-spp. Finally, if $V, W \subseteq A_d$ are $\ul{a}$-lpp and $\dim_K(V) \leq \dim_K(W)$, then $V \subseteq W$.

Let $L \subseteq \ov{A}_d$ be a lex-segment and $V=L+\ov{\aa}_d$. If $V\ne \ov{A}_d$, there exists the largest monomial $u \in \ov{A}_d \smallsetminus V$ with respect to the lexicographic order. In this case, we let 
$V^+ = V+ \langle u \rangle$; otherwise, we let $V^+ = V = A_d$. Either way, $V^+$ can be written as $L' + \ov{\aa}_d$, where $L'$ is a lex-segment, and therefore it is $\ul{a}$-lpp. 

\noindent If $V \ne \ov{\aa}_d$ we may write $V=W  \oplus \ov{\aa}_d$, with $W \ne 0$ a vector space minimally generated by monomials $m_1 \geq m_2 \geq \ldots \geq m_t$. In this case, we let $V^- = \langle m_1,\ldots,m_{t-1} \rangle + \ov{\aa}_d$; otherwise, we set $V^-=V = \ov{\aa}_d$. 

\medskip

The notion of segment we recall next is extracted from \cite{CaKu1}, and it will be crucial in the proof of Theorem \ref{cleli}. 
\begin{definition}
Let $V \subseteq A_d$ be a $K$-vector space, written as $V= \bigoplus_{i=0}^d V_{[d-i]}x_n^i$. Then, $V$ is called an \emph{$\ul{a}$-segment}, or simply a \emph{segment}, if it is $\ul{a}$-spp and, for all $i$,
\begin{enumerate}[(i)]
\item $V_{[i]} \subseteq \ov{A}_i$ is $\ov{\ul{a}}$-lpp, and
\item $V_{[i+j]} \subseteq \ov{\m}_{j}(V_{[i]})^+  + \ov{\aa}_{i+j}$ \quad for all $1 \leq j \leq d-i$.
\end{enumerate}
\end{definition}

Note that, if $V \subseteq A_d$ is $\ul{a}$-lpp, then it is an $\ul{a}$-segment.

\begin{remark} \label{Remark segment d+1} If $V \subseteq A_d$ is a segment, it immediately follows from the definition that $\m_1 V + \aa_{d+1} \subseteq A_{d+1}$ is also an $\ul{a}$-segment.
\end{remark} 


\begin{lemma} \label{Lemma segments comparable}
Let $V$ and $W$ be two $\ul{a}$-segments in $A_d$. Then either $V \subseteq W$, or $W \subseteq V$.
\end{lemma}
\begin{proof}
Write $V= \bigoplus_{i=0}^d V_{[d-i]} x_n^{i}$ and $W = \bigoplus_{i=0}^d W_{[d-i]} x_n^{i}$. If the conclusion is false, since $V$ and $W$ are segments we can find $i \ne j$ such that $V_{[i]} \subsetneq W_{[i]}$ and $V_{[j]} \supsetneq W_{[j]}$; say $j<i$. Since $V_{[j]}$ is lpp, $V_{[j]} \supseteq (W_{[j]})^+$, and therefore $V_{[i]} = V_{[i]}+\ov{\aa}_i \supseteq \ov{\m}_{i-j} V_{[j]} + \ov{\aa}_i \supseteq \ov{\m}_{i-j}(W_{[j]})^+  + \ov{\aa}_i \supseteq W_{[i]}$, which is a contradiction.
\end{proof}

\begin{definition} Let $V \subseteq A_d$ be a $K$-vector space, written as $V=\bigoplus_{i=0}^d V_{[d-i]}x_n^i$. We define the dimension sequence $\ul{\delta}(V) = (\dim_K(V_{[d]}),\dim_K(V_{[d]} \oplus V_{[d-1]}),\ldots,\dim_K(V)) \in \NN^{d+1}$. On the set of all such sequences, we consider the partial order given by point-wise inequality.
\end{definition}

\begin{lemma} \label{Lemma segments minimal} Let $V \subseteq A_d$ be an $\ul{a}$-spp $K$-vector space, written as $V = \bigoplus_{i=0}^d V_{[d-i]}x_n^i$. Assume that
\begin{enumerate}[(i)]
\item $V_{[i]} \subseteq \ov{A}_i$ is $\ov{\ul{a}}$-lpp for all $i$, and
\item $\ul{\delta}(V)$ is minimal among all dimension sequences of $\ul{a}$-spp $K$-vector subspaces $W = \bigoplus_{i=0}^d W_{[d-i]}x_n^i \subseteq A_d$ such that $\dim_K(W) = \dim_K(V)$ and $W_{[i]}$ is $\ov{\ul{a}}$-lpp for all $i$.
\end{enumerate}

\smallskip
\noindent
Then, $V$ is a segment.
\end{lemma}
\begin{proof}
  Assume that  $V$ is not a segment; then, there exist $i<j$ such that $V_{[i]} \subsetneq \ov{A}_i$ and $\ov{\m}_{j-i} (V_{[i]})^+ + \ov{\aa}_j \not\supseteq V_{[j]}$, and choose $i$ and $j$ so that $j-i$ is minimal. Observe that necessarily $i \geq \max\{d-a_n+1,0\}$, since otherwise $V_{[i]} = \ov{A}_i$. Since $V_{[i]}$ and $V_{[j]}$ are $\ov{\ul{a}}$-lpp, the fact that $\ov{\m}_{j-i} (V_{[i]})^+ + \ov{\aa}_j$ does not contain $V_{[j]}$ implies that $\ov{\m}_{j-i} (V_{[i]})^+ + \ov{\aa}_j$ is properly contained in $V_{[j]}$. In particular, the latter properly contains $\ov{\aa}_j$, and we have that
  \begin{equation}\label{aux1} \ov{\m}_{j-i} (V_{[i]})^+ + \ov{\aa}_j \subseteq (V_{[j]})^- .
  \end{equation}

  \medskip
  \noindent
  Now, define $W = \bigoplus_{k=0}^d W_{[d-k]} x_n^k$, where $W_{[i]} = (V_{[i]})^+$, $W_{[j]} = (V_{[j]})^-$, and $W_{[k]} = V_{[k]}$ for all $k \ne j,i$. We claim that $W$ is an $\ul{a}$-spp vector space.

  \smallskip
  \noindent
  In fact, let $k \geq \max\{d-a_n+1,0\}$; by stability, if $k \ne j,j-1,i,i-1$, then  $\ov{\m}_1 W_{[k]} = \ov{\m}_1 V_{[k]}  \subseteq V_{[k+1]} = W_{[k+1]}$; if $k=j$, then $\ov{\m}_1 W_{[j]} \subseteq \ov{\m}_1 V_{[j]} \subseteq V_{[j+1]} = W_{[j+1]}$ and  if $k=i-1$, then $\ov{\m}_1 W_{[i-1]} = \ov{\m}_1 V_{[i-1]} \subseteq V_{[i]} \subseteq W_{[i]}$.

  \noindent By costruction, we have that $\ov{\m}_{j-i} W_{[i]} \subseteq W_{[j]}$, see \eqref{aux1}; therefore, if $j-i=1$ we are done, again by stability.

  \noindent
  Thus, we may assume that $j-i>1$ and prove next  that $\ov{\m}_{k-i} W_{[i]} + \ov{\aa}_k= W_{[k]}$ for all $i<k< j$. Since $j-i$ is minimal, we have that $\ov{\m}_{k-i} W_{[i]} + \ov{\aa}_k= \ov{\m}_{k-i}(V_{[i]})^+ + \ov{\aa}_k \supseteq V_{[k]} = W_{[k]}$. If the containment were strict, then we would have $\ov{\m}_{k-i}(V_{[i]})^+ +\ov{\aa}_k \supseteq (V_{[k]})^+$ and, again by minimality, $\ov{\m}_{j-k} (V_{[k]})^+ +\ov{\aa}_j\supseteq V_{[j]}$; this would in turn imply $\ov{\m}_{j-i}(V_{[i]})^+ +\ov{\aa}_j = \ov{\m}_{j-k}\left(\ov{\m}_{k-i}(V_{[i]})^++\ov{\aa}_k\right) + \ov{\aa}_j \supseteq V_{[j]}$, contradicting our initial assumption on $i$ and $j$.

  \noindent The only case left to be shown is now $\ov{\m}_1W_{[j-1]} \subseteq W_{[j]}$. By applying what we have proved above for $k=j-1$,  we have that $\ov{\m}_1W_{[j-1]} + \ov{\aa}_j= \ov{\m}_1 \left(\ov{\m}_{j-1-i}W_{[i]} \right) +\ov{\aa}_j= \ov{\m}_{j-i}W_{[i]} +\ov{\aa}_j\subseteq W_{[j]}$, as desired.  

  \medskip
Thus, $W$ is  an $\ul{a}$-spp vector space; furthermore, it is clear from  definition that each $W_{[i]}$ is an $\ov{\ul{a}}$-lpp.
Finally, observe that $\ul{\delta}(W) < \ul{\delta}(V)$ by construction,  which  contradicts the minimality of $\ul{\delta}(V)$, and we are done. 
\end{proof}

\begin{proposition}  \label{Prop unique segment}
For every $d \geq 0$ and every $D \leq \dim_K(A_d)$ there exists a unique segment $V$ with $\dim_K(V) = D$. Moreover, the sequence $\ul{\delta}(V)$ is the minimum of the set of all sequences $\ul{\delta}(W)$ of $\ul{a}$-spp vector spaces $W = \bigoplus_{i=0}^d W_{[d-i]} x_n^i \subseteq A_d$ which have dimension $D$ and such that each $W_{[i]}$ is $\ov{\ul{a}}$-lpp.
\end{proposition}
\begin{proof}
By Lemma \ref{Lemma segments minimal} we have that any vector space with minimal dimension sequence is a segment, and by Lemma \ref{Lemma segments comparable} any two such segments are comparable, hence equal.  
\end{proof}



We already mentioned before that, if the EGH Conjecture held in full generality,  then $\LPP{\ul{a}}{I}$ would be the ideal with minimal growth among those containing a regular sequence of degree $\ul{a}$, and with Hilbert function equal to that of $I$. The proof is easy and we include it here. 

\begin{proposition}\label{Lemma lpp growth bis} Assume that EGH holds true. Let $I \subseteq A$ be a homogeneous ideal that contains a regular sequence of degree $\ul{a}$. Then $H(\m \LPP{\ul{a}}{I}) \leq H(\m I)$.
\end{proposition}
\begin{proof}
Let $d \geq 0$ be an integer, and let $\ul{a}' = (a_1,\ldots,a_r)$ be the degree sequence obtained from $\ul{a}$ by considering only the degrees $a_i$ such that $a_i \leq d$. Let $J=(I_d)$, and observe that $\LPP{\ul{a}}{I}_d = \LPP{\ul{a}'}{I}_d = \LPP{\ul{a}'}{J}_d$. Moreover, since $J_{d+1} = \m_1 I_d$, we have $H(\LPP{\ul{a}'}{J};d+1) = H(J;d+1) = H(\m I;d+1)$. Since $\m_1 \LPP{\ul{a}'}{J}_d \subseteq \LPP{\ul{a}'}{J}_{d+1}$, we finally obtain that $H(\m \LPP{\ul{a}}{I};d+1) = H(\m \LPP{\ul{a}'}{J}; d+1) \leq H(\m I;d+1)$.
\end{proof}

\noindent
We would like to observe that, even if we do not know that EGH holds in general,  we can still get an minimal growth statement in a Clements-Lindstr{\"o}m ring $A/\aa$, under milder hypotheses.

\begin{lemma}[Minimal Growth] \label{Lemma lpp growth} Assume that every homogeneous ideal containing $\aa$ satisfies $\EGHHH{\ul{a}}$. If $\aa\subseteq I \subseteq A$ is such an ideal,  then $H(\m \LPP{\ul{a}}{I} + \aa) \leq H(\m I + \aa)$.
\end{lemma}
\begin{proof}
Fix an integer $d \geq 0$, and let $J = (I_d)+\aa$. Note that both $I$ and $J$ satisfy the EGH, and $\LPP{\ul{a}}{I}_d = \LPP{\ul{a}}{J}_d$. Observe that $J_{d+1} = \m_1 J_d + \aa_{d+1} = \m_1 I_d + \aa_{d+1}$, and accordingly  $H(\LPP{\ul{a}}{J};d+1) = H(J;d+1) = H(\m I+\aa;d+1)$. Now, since $(\m\LPP{\ul{a}}{J} + \aa)_{d+1} = \m_1(\LPP{\ul{a}}{J})_d + \aa_{d+1} \subseteq (\LPP{\ul{a}}{J})_{d+1}$, we may conclude that $H(\m \LPP{\ul{a}}{I} + \aa;d+1) \leq H(\LPP{\ul{a}}{J} ;d+1) = H(\m I+\aa;d+1)$.
\end{proof}
We are finally in a position to prove the main result of this section. The simple idea underlying the new proof we present here is to demonstrate Clements-Lindstr\"om Theorem using Strong Hyperplane Restriction, like Green  proved Macaulay Theorem using generic hyperplane section; this also motivates why Part (ii) has been assimilated into the statement.

\begin{proof}[Proof of Theorem \ref{cleli}]
By adding sufficiently large powers of the variables $x_{r+1},\ldots,x_n$, we may assume that $r=n$. After taking any initial ideal, and by Proposition \ref{Proposition stabilize}, we may assume that $I$ is an $\ul{a}$-spp monomial ideal. By induction, we may also assume that both Part (i) and Part (ii) hold true in polynomial rings with less than $n$ variables, since the case $n=1$ is trivial. In particular, any lpp ideal of $\ov{A}$ has Minimal Growth, see Lemma \ref{Lemma lpp growth}. 

We write $I= \bigoplus_{i \geq 0} I_{[i]} x_n^i$; for all $i$, we let $J_{[i]} = \LPP{\ul{a}}{I_{[i]}}$,  which by induction 
is an ideal of $\ov{A}$. Next, we prove that $$J = \bigoplus_{i \geq 0} J_{[i]} x_n^i$$ is also an $\ul{a}$-spp ideal. First of all, observe that $I_{[k]} \subseteq I_{[k+1]}$ for all $k$, since $I$ is an ideal. This implies that $H(J_{[k]}) =H(I_{[k]}) \leq H(I_{[k+1]}) = H(J_{[k+1]})$. Since the ideals $J_{[k]}$ and $J_{[k+1]}$ are lpp, it follows that $J_{[k]} \subseteq J_{[k+1]}$, which, in turn, translates into $J$ being an ideal.
  Since  $I$ is $\ul{a}$-spp, for all $i < a_n-1$ we have $\ov{\m}_1 I_{[i+1]} \subseteq I_{[i]}$ and $\ov{\aa} \subseteq I_{[i]}$; thus 
\[
H(J_{[i]}) = H(I_{[i]}) \geq H(\ov{\m}_1 I_{[i+1]}+\ov{\aa}) \geq H(\ov{\m}_1 J_{[i+1]}+\ov{\aa}),
\]
where the last inequality follows from Lemma \ref{Lemma lpp growth}. This yields that  that $\ov{\m}_1 J_{[i+1]} \subseteq J_{[i]}$ for all $i < a_n-1$, and $J$ is $\ul{a}$-spp by stability.

\smallskip
Given an $\ul{a}$-spp vector space $V \subseteq A_d$, denote by $\sigma(V)$ the segment contained in $A_d$ which has the same dimension as $V$. Let $J=\bigoplus\limits_{d\geq 0} J_d$ be the homogeneous ideal we constructed above and let
$$\sigma(J) = \bigoplus_{d \geq 0} \sigma(J_d).$$
We claim that $\sigma(J)$ is the $\underline{a}$-lpp ideal we are looking for.

\medskip
\noindent
First of all we show that it is an ideal. Fix a degree $d \geq 0$, and write $J_d = \bigoplus_{i=0}^d (J_d)_{[d-i]} x_n^i$, $\sigma(J_d) = \bigoplus_{i=0}^d \sigma(J_d)_{[d-i]} x_n^i$; for notational simplicity, in the following we let $\sigma_{[d-i]}=\sigma(J_d)_{[d-i]}$.
By stability, we then have
\[
\m_1 J_d +\aa_{d+1} = \begin{cases} \left(\ov{\m}_1(J_d)_{[d]} +\ov{\aa}_{d+1}\right) \oplus \left(\bigoplus_{i=0}^{d} (J_d)_{[d-i]}x_n^{i+1}\right), & {\hspace{-.35cm}\rm if } \ d < a_n-1, \\ \\
  \left(\ov{\m}_1(J_d)_{[d]} + \ov{\aa}_{d+1}\right) \oplus \left(\bigoplus_{i=0}^{a_n-2} (J_d)_{[d-i]}x_n^{i+1}\right)  \oplus \left(\bigoplus_{i=a_n}^d \ov{A}_{d-i} x_n^{i} \right),& {\hspace{-.35cm}\rm if } \ d \geq a_n-1,
\end{cases}
\]
and \[
\m_1 \sigma(J_d) + \aa_{d+1} = \begin{cases} \left(\ov{\m}_1\sigma_{[d]} + \ov{\aa}_{d+1} \right) \oplus \left(\bigoplus_{i=0}^{d} \sigma_{[d-i]}x_n^{i+1}\right), & {\rm if } \ d < a_n-1, \\ \\
\left(\ov{\m}_1\sigma_{[d]} + \ov{\aa}_{d+1}\right)  \oplus \left(\bigoplus_{i=0}^{a_n-2} \sigma_{[d-i]}x_n^{i+1}\right)  \oplus \left(\bigoplus_{i=a_n}^d \ov{A}_{d-i} x_n^{i} \right),& {\rm if } \ d \geq a_n-1.
\end{cases}
\]

When $d<a_n-1$, we set $\sigma_{[a_n-1]} = (J_d)_{[a_n-1]}=0$. From the above equalities we thus get
\begin{equation}\label{auxCL1}\begin{split}
\dim_K(\m_1 J_d+\aa_{d+1})&  - \dim_K(\m_1 \sigma(J_d)+\aa_{d+1})  =\\
&=  \left(\dim_K(\ov{\m}_1(J_d)_{[d]}+ \ov{\aa}_{d+1}) - \dim_K(\ov{\m}_1(\sigma_{[d]}+\ov{\aa}_{d+1}))\right)+\\ &\,\,\,\,\,\,\, \left(\dim_K(\sigma_{[a_n-1]}) - \dim_K((J_d)_{[a_n-1]})\right). 
  \end{split}
\end{equation}

Since $\sigma(J_d)$ is a segment, $\sigma_{[d]} \subseteq \ov{A}$ is $\ov{\ul{a}}$-lpp and its dimension sequence $\ul{\delta}=\ul{\delta}(\sigma(J_d))$ is minimal for the Proposition \ref{Prop unique segment}. Moreover, the $\ov{\ul{a}}$-lpp vector space $L_d\subseteq \ov{A}_d$ with the same Hilbert function as $(J_d)_{[d]}$ has Minimal Growth, and $\sigma_{[d]} \subseteq L_d$ by the minimality of $\ul{\delta}$. Therefore,
\begin{equation*}
\dim_K(\ov{\m}_1(J_d)_{[d]}+\ov{\aa}_{d+1}) \geq \dim_K(\ov{\m}_1 L_d+\ov{\aa}_{d+1})\geq  \dim_K(\ov{\m}_1 \sigma_{[d]}+\ov{\aa}_{d+1}) .
\end{equation*}

\noindent
Recall that the last entry of the dimension sequence is the dimension of the vector space itself; thus, since $\sigma(J_d)$ and $J_d$ have the same dimension and $\ul{\delta}(J_d)\geq \ul{\delta}$ we get $\dim_K(\sigma_{[a_n-1]})\geq \dim_K((J_d)_{[a_n-1]})$.  An application of \eqref{auxCL1} now yields $$\dim_K(\m_1 J_d + \aa_{d+1}) \geq \dim_K(\m_1 \sigma(J_d)+\aa_{d+1}).$$

\noindent
 Since $J$ is an ideal that contains $\aa$, we have that $\m_1 J_d + \aa_{d+1} \subseteq J_{d+1}$ and, thus,
$$\dim_K(\m_1 \sigma(J_d) + \aa_{d+1}) \leq  \dim_K(J_{d+1}) = \dim_K(\sigma(J_{d+1})).$$ By Remark \ref{Remark segment d+1}, we know that $\m_1 \sigma(J_{d}) + \aa_{d+1}$ is a segment, and so is $\sigma(J_{d+1})$ by definition; then, it follows that $\m_1 \sigma(J_d) \subseteq \sigma(J_{d+1})$. We may finally conclude that $\sigma(J)$ is an ideal, which is $\ul{a}$-spp by construction, and has the same Hilbert function as $I$.
 
\medskip
Next, we observe that $\sigma(J)$  satisfies Part (ii) of the theorem, since $H(\sigma(J) + (x_n^i);d)$ is just the $i$-th entry of $\ul{\delta}(\sigma(J_d))$, $H(J+(x_n^i);d)$ is the $i$-th entry of $\ul{\delta}(J_d)$, and $\ul{\delta} \leq \ul{\delta}(J_d)$.

By construction, $\sigma(J)$ is the ideal with all the required properties, once we have proved the following claim.

\begin{claim*} $\sigma(J)$ is $\ul{a}$-lpp. 
\end{claim*}

{\em Proof of the Claim}: By contradiction, there exists a degree $d$ such that $\sigma(J_d)$ is an $\ul{a}$-spp $D$-dimensional vector space which is not lpp; thus, we may consider a counterexample of degree $d$ and of minimal dimension $D$ for which the operator $\sigma$ does not return an $\ul{a}$-lpp vector space of dimension $D$ inside $A_d$; then, if we apply $\sigma$ to any $(D-1)$-dimensional $\ul{a}$-spp vector space of $A_d$, we obtain an $\ul{a}$-lpp vector space, but there is an $\ul{a}$-spp vector space of dimension $D$ which is transformed by $\sigma$ into an $\ul{a}$-spp vector space $V+ \langle v \rangle$ which is not lpp. Thus, $V$ is $\ul{a}$-lpp, $V+\langle v \rangle$ is $\ul{a}$-segment, and we write them as  $$V = \bigoplus_{i=0}^d V_{[d-i]}x_n^i,\,\,\,\,\,\,\,V+\langle v \rangle = \bigoplus_{i=0}^d \widetilde{V}_{[d-i]} x_n^i.$$
Let also $w$ be the monomial such that $V+\langle w\rangle$ is the $\ul{a}$-lpp vector subspace of dimension $D$ of $A_d$ and observe that  $w > v$. Write $v= \ov{v} x_n^t$ and $w = \ov{w} x_n^s$, where $\ov{v},\ov{w}$ are monomials in $\ov{A}$.

\noindent
 Since $V+ \langle v \rangle$ is a segment, we have that $t \geq s$.

\noindent
If $t=s$ we immediately get a contradiction, since by construction 
$\ov{v}$ and $\ov{w}$ would both be the largest monomial of degree $d-t$ which is not contained in  
$V_{[d-t]}$.

\smallskip
Therefore, we may assume that $t>s$, and $a=\deg(\ov{w})=d-s > d-t=\deg(\ov{v})=b$. Observe that  
$\ov{v} \in \widetilde{V}_{[d-t]}$, and that $d-t < a_n$. Moreover
$\ov{\m}_{a-b}\widetilde{V}_{[d-t]} \subseteq \widetilde{V}_{[d-s]}$ holds by stability applied to $V+\langle v \rangle$. 
We write $\ov{w} = x_{i_1} \cdots x_{i_a}$ and $\ov{v} = x_{j_1} \cdots x_{j_b}$, with $i_1\leq \ldots \leq i_a$ and $j_1\leq \ldots\leq j_b$. Since $w > v$ we have two cases, either 
$\ov{v}$ divides $\ov{w}$, or $x_{i_1} \cdots x_{i_b} > \ov{v}$. In both cases, it is easy to see that $\ov{w} \in \ov{\m}_{a-b} \widetilde{V}_{[d-t]} \subseteq \widetilde{V}_{[d-s]}$, 
and thus $w \in V+\langle v\rangle$, which is a  contradiction.  \qedhere
\end{proof}

The proofs of Theorem \ref{cleli} (i) previously available in the literature do not include Part (ii), the Strong Hyperplane Section of Gasharov. One advantage of our approach is that, with little additional effort, one can show that the Betti numbers of an $\ul{a}$-spp ideal are at most those of the corresponding $\ul{a}$-lpp ideal; see \cite{Mu,CaKu2}. Furthermore, combining this fact with Remark \ref{Remark Betti stability}, one recovers the LPP-Conjecture for ideals containing pure-powers ideals in characteristic zero, which is the main result of 
\cite[Section 3]{MeMu}. 
Note that, in \cite{MeMu}, the authors also provide a characteristic-free proof that settles the LPP-Conjecture for ideals that contain pure-powers.


\section{Artinian reduction and linkage} \label{Section Artinian + Linkage}

In this brief section we collect some results which will be useful in what follows. We start with Proposition 10 in \cite{CaMa}, which offers in many cases a way to prove the EGH Conjecture in the Artinian case only.

\begin{proposition} \label{Prop Artinian reduction}
  Let ${\bf f} \subseteq A=K[x_1,\ldots,x_n]$ be an ideal generated by a regular sequence of degree $\ul{a}$, and $\ell$ be a linear $A/{\bf f}$-regular form. Let also $\ov{A} = A/(\ell)$, and ${\bf \ov{f}}={\bf f}\ov{A}$. If every homogeneous ideal of $\ov{A}$ containing ${\bf \ov{f}}$ satisfies $\EGHHH{\ov{\ul{a}}}$, then every homogeneous ideal of $A$ containing ${\bf f}$ satisfies $\EGHHH{\ul{a}}$.
\end{proposition}

\begin{proof} 
 Let  $I \subseteq A$ be a homogeneous ideal that contains ${\bf f}$ 
and for $i\geq 0$ we let 
$I_i = (I:_A \ell^i) + (\ell)$. 
By assumption,  there exist $\ov{\ul{a}}$-lpp ideals $J_i \subseteq K[x_1,\ldots,x_{n-1}]$ with the same Hilbert function as $I_i/(\ell)$. Now, we define $$J = \bigoplus_{i \geq 0} J_i x_n^i, 
$$ 
and we claim that $J$ is an ideal with the same Hilbert function as $I$; since $\aa \subseteq J_0 \subseteq J$, the conclusion will then follow from Theorem \ref{cleli}. 

\noindent By considering the short exact sequences 
  $0 \lra A/(I:_A \ell^j)(-1) \stackrel{\cdot \ell}{\lra}   A/(I:_A \ell^{j-1}) \lra  \ov{A}/I_{j-1}\overline{A} \lra 0$  for all $j$, a straightforward computation yields that $H(J) = H(I)$.

 
What it is left to be shown is  that $J$ is an ideal. 
Let as before $\ov{\m} = (x_1,\ldots,x_{n-1})$; since $J_i$ is an  ideal of $\ov{A}$, we have  $\ov{\m} J_i \subseteq J_i$ for all $i$ and, accordingly, $\ov{\m}J\subseteq J$. The condition $x_n J \subseteq J$ translates into the containments $J_i \subseteq J_{i+1}$ for all $i\geq 0$. Since each $J_i$ is an $\ov{\ul{a}}$-lpp ideal, it suffices to show that $H(J_i) \leq H(J_{i+1})$, which holds true since $I_i  \subseteq I_{i+1}$. 
\end{proof}


We now recall some results from the theory of linkage.  In Section \ref{Section CL} we introduced the following notation: given a homogeneous ideal $I \subseteq A=K[x_1,\ldots,x_n]$ containing an ideal ${\bf f}$ generated by a regular sequence of degree $\ul{a} = (a_1,\ldots,a_n)$, we let $I_{{\bf f}}^\ell = ({\bf f}:_AI)$, and call it the {\it link of $I$ with respect to ${\bf f}$}, which is an ideal that contains ${\bf f}$. Obviously, the link depends on ${\bf f}$; however, when it is clear from the context which ${\bf f}$ we consider, we denote $I^\ell_{\bf f}$ simply by $I^\ell$. 


\begin{proposition}  \label{Prop link}
Let $\ul{a} = (a_1,\ldots,a_n)$ and $A$, $I$, ${\bf f}$ be as above; let also  $R=A/{\bf f}$ and  $s=\sum_{i=1}^n(a_i-1)$. Then,
\begin{enumerate}
\item $(I^\ell)^\ell = I$.
\item  $H(IR;d) = H(R;d) - H(I^{\ell}R; s-d)$.
\item ${\rm type\,}(R/IR)=\mu(I^\ell R)$, i.e., the dimension of the socle of $R/I R$ equals the minimal number of generators of its linked ideal.
\end{enumerate}
In particular, if $I=({\bf f}+(g))$ is an almost complete intersection, then the ideal $I^\ell = ({\bf f}:_A g)$ defines a Gorenstein ring, and viceversa. Moreover, if $\deg(g)=D$, then ${\rm soc}(({\bf f}:_A g) R)$ is concentrated in degree $s-D$.
\end{proposition}

\begin{proof}
Observe that the functor $(-)^\vee = \Hom_R(-,R)$ is the Matlis dual, since $R$ is Gorenstein Artinian. The statements that we want to prove are a direct consequence of Matlis duality, see \cite[Sections 3.2 and 3.6]{BrHe}. It is well known that a module and its Matlis dual have the same annihilator. In particular, since $(A/I)^\vee \cong I^\ell/{\bf f}$, we obtain that $I = \ann_A(A/I) = \ann_A(I^\ell/{\bf f}) = (I^\ell)^\ell$, which proves (i). For (ii), recall that in the graded setting one has $((A/I)^\vee)_d \cong (A/I)_{s-d}$, for all $d \in \ZZ$. Since $(A/I^\ell)^\vee \cong I/{\bf f}$, the claim follows from the graded short exact sequences of $K$-vector spaces $0 \to (I/{\bf f})_d \to (A/{\bf f})_d \to (A/I)_d \to 0$. Part (iii) is again a consequence of Matlis duality. 
\end{proof}


We conclude this part with an easy lemma. 

\begin{lemma} \label{Lemma spara gradi}
Let $\ul{a} = (a_1,\ldots,a_r)$ and $\ul{b} = (b_1,\ldots,b_r)$ be degree sequences satisfying $a_i \leq b_i$ for all $i=1,\ldots,r$. If an ideal $I$ satisfies $\EGHHH{\ul{a}}$, then it satisfies $\EGHHH{\ul{b}}$.
\end{lemma}
\begin{proof}
By assumption, $J=\LPP{\ul{a}}{I}$ is a $\ul{a}$-lpp ideal with the same Hilbert function as $I$. By our assumption on the degree sequences, $J$ also contains the pure-powers ideal $(x_1^{b_1},\ldots,x_r^{b_r})$. Therefore, by Theorem \ref{cleli}, $\LPP{\ul{b}}{J}$ is a $\ul{b}$-lpp ideal with the same Hilbert function as $I$.
\end{proof}

\section{Results on the EGH conjecture}\label{Section EGH}

We collect in the following the most relevant cases when EGH is known to be true. We start with a very recent result, Theorem \ref{Thm EGH large degrees}, proved by the first two authors in \cite[Theorem A]{CaDeS1}, which improves an older result due Maclagan and the first author, \cite[Theorem 2]{CaMa}. Indeed, as we show in this section, from Theorem \ref{Thm EGH large degrees} one can derive with little
effort all of the significant known cases of the EGH Conjecture which take into account only hypotheses on the degree sequence $\ul{a}$ and not on the ideal $I$. 
A further generalization can be found in \cite{CaDeS1}, see Theorem 3.6.

\begin{theorem} \label{Thm EGH large degrees}
Let $I \subseteq A$ be a homogeneous ideal which contains a regular sequence of degree $\ul{a} = (a_1,\ldots,a_r)$ and assume that $a_i \geq \sum_{j=1}^{i-1} (a_j-1)$ for all $i \geq 3$; then, $I$ satisfies $\EGHHH{\ul{a}}$.
\end{theorem}
\normalcolor

\begin{proof}
For brevity's sake, we present here only the proof  of the weaker statement \cite[Theorem 2]{CaMa}, that is, we will assume that $a_i > \sum_{j=1}^{i-1} (a_j-1)$ for all $i \geq 3$. Observe that, by Proposition \ref{Prop Artinian reduction}, we may let $r=n$ and work by induction on $n$. Let $\ov{\ul{a}}=(a_1,\ldots,a_{n-1})$; by induction, suppose that  
every ideal of $A$ containing a regular sequence of degree $\ov{\ul{a}}$ satisfies $\EGHH{{\ov{\ul a}}}{d}$ for all $d$. 

  Clearly, for $d<a_n-1$, we have that $\EGHH{\ul{a}}{d}$  is equivalent to $\EGHH{\ov{\ul{a}}}{d}$. Thus, let $d+1 \geq a_n$, so that $s-(d+1) < a_n-1$; by induction, $I^\ell$ satisfies $\EGHHH{\ov{\ul{a}}}$ and the previous case yields that  $I^\ell$ satisfies $\EGHH{\underline{a}}{s-(d+1)}$ for all $d+1 \geq a_n$. By Proposition \ref{Prop link} (ii), we know that\, $H(IR;d) = H(R;d) - H(I^{\ell}R;s-d)$,\, where\, $R=A/{\bf f}$\, and\, $s=\sum_{i=1}^n (a_i-1)$. It now follows that $I$ satisfies $\EGHH{\ul{a}}{d}$ also for all $d+1 \geq a_n$, and the proof is complete. 
\end{proof}



As we have already observed in Remark \ref{Remark three proofs} (1), Theorem \ref{Thm EGH large degrees} yields the EGH for $r \leq 2$.

\smallskip

\noindent One big advantage of Theorem \ref{Thm EGH large degrees} is that it can be applied in order to obtain growth bounds for the Hilbert function which are at least as good as the ones given by Macaulay Theorem. 
This can be done for \emph{any} homogenous ideal, regardless of the degree sequence.
The key observation to see this is the following.



\begin{lemma} \label{Lemma higher degree}
Assume that $|K|=\infty$ and that $I$ contains  an ideal ${\bf f}$ generated by a regular sequence  of degree $\un{a}=(a_1,\ldots,a_r)$. If $\un{b}=(b_1,\ldots,b_r)$ is a degree sequence such that $b_i \geq a_i$ for all $i$, then $I$ contains an ideal {\bf g} generated by a regular sequence of degree $\un{b}$.
\end{lemma}


\begin{proof}
  We proceed by induction on $r \geq 1$. Let $r=1$ and  observe that $I_{b_1} \ne 0$ since $b_1 \geq a_1$. It follows that there exists a regular element $g_1\in I$ of degree $b_1$.

  By induction, we have constructed a homogeneous ideal ${\bf g'}=(g_1,\ldots,g_{r-1})$, which is unmixed and generated by a regular sequence of degrees $b_1,\ldots,$ $b_{r-1}$.  Observe that, since $I$ contains\, $f_1,\ldots,f_r$,\, we  have that   $\Ht(I_jA) \geq r$\, for all\, $j \geq a_r$. In particular, the ideal ${\bf g'} + I_{b_r}A$ has height at least $r$, since $b_r \geq a_r$. Thus, by prime avoidance, we  find an element $g_r \in I_{b_r}$ which is regular modulo ${\bf g'}$\, and\, ${\bf g}=(g_1,\ldots,g_r)$ is the ideal we were looking for.
\end{proof}



As another application of the theory of linkage to the EGH Conjecture, we now present a result due to Chong \cite{Cho}, which settles the conjecture for Gorenstein ideals of height three. 

\begin{proposition} \label{Proposition Chong}
Let $I$ be a homogeneous ideal that contains an ideal ${\bf f}$ generated by a regular sequence of degree $\ul{a} = (a_1,\ldots,a_n)$. Assume that $\ul{b} = (b_1,\ldots,b_n)$ is a degree sequence such that $b_i \leq a_i$ for all $i$, and\, $I_{\bf f}^\ell$\, satisfies $\EGHHH{\ul{b}}$;\, then $I$ satisfies $\EGHHH{\ul{a}}$.
\end{proposition}
\normalcolor



\begin{proof}
  Let $s=\sum_{i=1}^n (a_i-1)$\, and\,  $I^\ell=I^\ell_{\bf f}$;\,  by hypothesis  there exists a $\ul{b}$-lpp ideal $J$ with the same Hilbert function as $I^\ell$ that also contains the pure-powers ideal $\aa=(x_1^{a_1},\ldots,x_r^{a_r})$,  since $a_i \geq b_i$ for all $i$. Consider now $J^\ell_\aa $;\, by Proposition \ref{Prop link} (ii)  for all $d \geq 0$ we have
\[
H(I/{\bf f};d) = H(A/{\bf f};d) - H(I^{\ell}/{\bf f};s-d) = H(A/\aa;d) - H(J/\aa;s-d) = H(J^{\ell}_{\aa}/{\aa};d).
\]
By Theorem \ref{cleli}, there exists an $\ul{a}$-lpp ideal with the same Hilbert function as $J^\ell_\aa$, and we are done.
\end{proof}


Observe that in the above proof  we used Theorem \ref{cleli} to transform the monomial ideal $J^\ell_\aa$  into an $\ul{a}$-lpp ideal.  In fact, it can be proved in general  that $J^\ell_\aa$ is already $\ul{a}$-lpp whenever $J$ is $\ul{a}$-lpp,  see for instance \cite[Theorem 5.7]{RiSa}, or \cite[Proposition 3.2]{CaSa}. 


Sequentially bounded licci ideals  were first introduced in \cite{Cho}, and are those ideals to which  Proposition \ref{Proposition Chong} can be applied repeatedly in order to prove the EGH Conjecture. We recall the  main definitions here.


\begin{definition} \label{Defn licci}
Let $I \subseteq A=K[x_1,\ldots,x_n]$ be a homogeneous ideal, and set $I_0=I$. 
We say that $I$ is {\em linked to a complete intersection}, or {\em licci} for short, if there exist ideals $I_j = (I_{j-1})^\ell_{{\bf f}_{j}}$ where ${\bf f}_1,\ldots,{\bf f}_s$ are ideals of the same height as $I$ generated by regular sequences of degrees $\ul{a}_1,\ldots,\ul{a}_{s}$, such that $I_s$ is generated by a regular sequence of degree $\ul{a}_{s+1}$. 

  \noindent We say that $I$ is {\em sequentially bounded licci} if the above sequence also satisfies $\ul{a}_1 \geq \ldots \geq \ul{a}_{s+1}$. 
\end{definition}

\noindent We also recall that $I$ is said to be {\em minimally licci} if it is licci and, in addition, for each $j$ the regular sequence generating ${\bf f}_{j+1}$ can be chosen to be of minimal degree among all the regular sequences contained in $I_j$. Observe that ${\bf f}_j \subseteq I_j$, therefore minimally licci ideals are sequentially bounded licci. 
It was proved by Watanabe \cite{Wa} that height three Gorenstein ideals are licci. Later on, Migliore and Nagel show that such ideals are also minimally licci \cite{MiNa}. 
We see next how these facts together, combined with Proposition \ref{Proposition Chong}, yield the main result of \cite{Cho}.
\normalcolor


\begin{theorem} \label{Thm Chong} Let $I \subseteq A$ be a sequentially bounded licci ideal, where the first link of $I$ is performed with respect to a regular sequence of degree $\ul{a}$; then $I$ satisfies $\EGHHH{\ul{a}}$.

  In particular, if $I$ is a Gorenstein ideal of height 3 containing a regular sequence of degree $\ul{a}=(a_1,a_2,a_3)$, then $I$ satisfies $\EGHHH{\ul{a}}$.
\end{theorem}

\begin{proof}
We prove the first part only for  $n=r$, and we refer the reader to the original paper for the reduction to this case; this is shown in \cite[Proposition 10]{Cho}, where the proof runs along the same lines as that of  Proposition \ref{Prop Artinian reduction}.

\noindent
Since $I_s$ is a complete intersection of degree $\ul{a}_{s+1}$ by assumption, it trivially satisfies $\EGHHH{\ul{a}_{s+1}}$; therefore Proposition \ref{Proposition Chong} implies that $I_{s-1}$ satisfies $\EGHHH{\ul{a}_{s}}$, and its 
 repeated application to the sequence of linked ideals eventually yields that $I$ satisfies $\EGHHH{\ul{a}_1}$, that is $\EGHHH{\ul{a}}$. 
\end{proof}

\begin{remark} The height 3 Gorenstein case proved by Chong is also related to a previous result due to Geramita and Kreuzer concerning the Cayley-Bacharach Conjecture in $\PP^3$ \cite[Corollary 4.4]{GeKr}. In fact, EGH for a height 3 Gorenstein ideal $I$ is equivalent to EGH for its linked ideal $I^\ell$, which is an almost complete intersection by Proposition \ref{Prop link} (iii). As pointed out in the introduction, EGH for almost complete intersections implies the Cayley-Bacharach Conjecture \ref{Conj CB}. 
\end{remark}



Next, we present a result due to Francisco \cite[Corollary 5.2]{Fr} which settles $\EGHH{\ul{a}}{D}$ for almost complete intersections $({\bf f}+(g))$ in the first relevant degree, namely $D=\deg(g)$.


\begin{theorem} \label{Thm Franscisco}
Let ${\bf f} \subseteq A$ be an ideal generated by a regular sequence of degree $\ul{a}=(a_1,\ldots,a_r)$, and let $g \notin {\bf f}$ be an element of degree $D \geq a_1$ such that $I = {\bf f} + (g)$. Then, $I$ satisfies $\EGHH{\ul{a}}{D}$.
\end{theorem}

\begin{proof}
  We may assume that $K$ is infinite. First, we reduce to the Artinian case by arguing as follows: we choose some $N > D+1$ and homogeneous elements of degree $N$  such that $f_1,\ldots,f_r,f_{r+1},\ldots,f_n$ is a full regular sequence of degree $\ul{a}' = (a_1,\ldots,a_r,N,\ldots,N)$. In this way, proving $\EGHH{\ul{a}}{D}$ for $I$ is equivalent to proving $\EGHH{\ul{a}'}{D}$ for $I + (f_{r+1},\ldots,f_n)$. Thus, for the rest of proof  $r=n$ and $A/{\bf f}$ is Artinian.

  \smallskip
  Now, let $b$ be the unique integer such that $\sum_{i=1}^b (a_i-1) \leq D < \sum_{i=1}^{b+1} (a_i-1)$. It is then easy to see that $J = \aa + (h)$,\, where $h=x_1^{a_1-1} \cdots x_{b}^{a_{b}-1} \cdot x_{b+1}^{D-\sum_{i=1}^{b} (a_i-1)}$, is the smallest $\ul{a}$-lpp ideal with $H(J;D) = H(I;D)$.

  To conclude the proof, it suffices to show that $H(J;D+1) \leq H(I;D+1)$. To this end, let $s=\sum_{i=1}^n (a_i-1)$, and consider the links $I^\ell = I^\ell_{\bf f}=({\bf f}:_A I)$ and $J^\ell = J^\ell_\aa=(\aa:_A J)$. The natural graded short exact sequences $$0 \to A/I^\ell(-D) \to A/{\bf f} \to A/I \to 0\,\,\,\,\,\,  \text{and}\,\,\,\,\,\,0 \to A/J^\ell(-D) \to A/\aa \to A/J \to 0$$  show that we only have to  prove that $H(J^\ell;1) \geq H(I^\ell;1)$. From a direct computation we get   that  $$J^\ell = \aa:(h)=(x_1,\ldots,x_{b},x_{b+1}^{\sum_{i=1}^{b+1} (a_i-1)-D+1},x_{b+2}^{a_{b+2}},\ldots,x_n^{a_n}),$$ that is, $H(J^\ell;1) = b$.

  Suppose, by contradiction, that  $I^\ell$ contains $c$ linear forms, with $c>b$; then, by Prime Avoidance we can find a homogeneous ideal ${\bf g}\subseteq I^\ell$ generated by a regular sequence  of degree $(1,\ldots,1,a_{c+1},\ldots,a_n)$ such that the socle degree of $A/{\bf g}$\, is\, $\sum_{i=c+1}^n (a_i-1) < \sum_{i=b+1}^n (a_i-1) \leq s-D$. Thus,
$H(A/I^\ell;s-D) \leq H(A/{\bf g};s-D) = 0$ which is not possible, since the ring $A/I^\ell$ is Gorenstein of socle degree $s-D$ by Proposition \ref{Prop link} (iii). 
\end{proof}

\begin{remark}
It is easy to see by means of  Lemma \ref{Lemma higher degree} that the condition $D \geq a_1$ in the statement of Theorem \ref{Thm Franscisco} can always be met.
\end{remark}
\normalcolor


Observe that, again by Proposition \ref{Prop link} (ii), the statement of Theorem \ref{Thm Franscisco} is equivalent to proving $\EGHH{\ul{a}}{s-D-1}$ for the ideal $I^\ell=I^\ell_{\bf f}$. Since the socle of $A/I^\ell$ is concentrated in degree $s-D$, this  is equivalent to controlling the growth of the Hilbert function of a Gorenstein ring from socle degree minus 1 to the socle degree. For other results of this nature, see for instance \cite{Crelle}.

\normalcolor



The next result we present is due to Abedelfatah,  see \cite{Ab} and \cite{Ab1}; it can be viewed as a generalization of the Clements-Lindstr{\"o}m Theorem 
to ideals that contain a regular sequence generated by products of linear forms. Below we provide the proof of the general version, cf. \cite[Theorem 3.4]{Ab1}.


\begin{theorem}\label{Thm Abedelfatah}
Let ${\bf f} \subseteq A$ be an ideal generated by a regular sequence of degree $\ul{a} = (a_1,\ldots,a_r)$. Assume that ${\bf f} \subseteq P$, where $P$ is an ideal generated by products of linear forms. Then, any ideal $I \subseteq A$ that contains $P$  satisfies $\EGHHH{\ul{a}}$.
\end{theorem}

\begin{proof}
  By induction we may assume that the claim is true for ideals in polynomial rings with less than $n$ variables, since the base case $n=1$ is trivial.

  \noindent Let $s$ be the smallest degree of a minimal generator $p$ of $P$. Since $s \leq a_1$, by Lemma \ref{Lemma spara gradi} it suffices to show that $I$ satisfies $\EGHHH{\ul{a}'}$, where $\ul{a}'=(s,a_2,\ldots,a_r)$. Moreover, by Theorem \ref{cleli}, it is enough to prove that, for every degree $d \geq 0$, there exists a monomial ideal $J$ that contains $(x_1^{s},x_2^{a_2},\ldots,x_r^{a_r})$ such that $H(I;d) = H(J;d)$ and $H(I;d+1) = H(J;d+1)$. 

  \smallskip
  We write $p =  \ell_1 \cdots \ell_s$, where $\ell_i$ are linear forms which we  order as follows:

  \noindent For $k=1,\ldots,s$,\, let $I_k^{(0)}$ denote the image ideal of $I$ in $A/(\ell_k)$ and choose $\ell_1$ so that $H(I_1^{(0)};d)=\min_k \{H(I_k^{(0)};d) \}$.

  \noindent
  Inductively, given  $\ell_1,\ldots,\ell_j$, for $k=j+1,\ldots,s$ we let $I^{(j)}_k$ denote the image ideal  of $(I:_A (\ell_1 \cdots \ell_j))$\, in\, $A/(\ell_k)$ and choose $\ell_{j+1}$ so that $H(I^{(j)}_{j+1};d-j) =\min_k \{H(I_k^{(j)};d-j) \}$.

  \medskip
  Now, with some abuse of notation, we let  $A_k=A/(\ell_k)$ for $k=1,\ldots,s$; for notational simplicity, we also set $I_j = I^{(j)}_{j+1}$ for $j=0,\ldots,s-1$.
By construction, we thus have
  \begin{equation}\label{help4abe2}
    H(I_j;d-j) \leq H(I_{j+1};d-j)\,\,\,\text{ for all } j=0,\ldots,s-1.
  \end{equation}

  Moreover, for all $j=1,\ldots,s-1$, the short exact sequences 
\[
\xymatrix{ 
0 \ar[r] & A/(I:_A (\ell_1 \cdots \ell_j)) (-1) \ar[r] & A/(I:_A (\ell_1 \cdots \ell_{j-1}))\ar[r] & A_j/ I_{j-1} \ar[r] & 0
}
\]
provide that  \begin{equation}\label{help4abe}
  H(A/I;i) = \sum_{j=0}^{s-1} H(A_{j+1}/I_{j};i-j),\,\,\, \text{for all } i. 
\end{equation}

Let $\tilde{\ul{a}} = (a_2,\ldots,a_n)$ and $\tilde{A} = K[x_2,\ldots,x_n]$. Observe that $A_k\cong \tilde{A}$ for all $k$, thus, by induction, we can find $\tilde{\ul{a}}$-lpp ideals $J_{[j]}$ in $\tilde{A}$ with the same Hilbert function as $I_j$, for $j=0,\ldots,s-1$. Consider now $J = \bigoplus_{j=0}^{s-1} J_{[j]} x_1^j \oplus A x_1^s$, and let $J_d$ denote the degree $d$ component of $J$.  If we show, and we shall do, that $\m_1 J_d \subseteq J_{d+1}$, that is, $J$ is closed under multiplication from degree $d$ to degree $d+1$, then the proof is  complete, since  $H(A/J;i) = H(A/I;i)$ for all $i$\, by \eqref{help4abe}. 

\noindent To this end, we clearly have that $(x_2,\ldots,x_n)_1 (J_{[j]})_{d-j} \subseteq (J_{[j]})_{d-j+1}$, since each $J_{[j]}$ is an ideal in $\tilde{A}$. It is left to show that $x_1 J_d \subseteq J_{d+1}$, which translates into $(J_{[j]})_{d-j} \subseteq (J_{[j+1]})_{d-j}$ for all $j=0,\ldots,s-1$; since such ideals are both $\tilde{\ul{a}}$-lpp, this is yielded  by \eqref{help4abe2}.
\end{proof}

\begin{corollary} The EGH Conjecture is true for monomial ideals.
\end{corollary}


Another interesting known case, of different nature, is when the regular sequence that defines ${\bf f}$ is a Gr{\"o}bner basis with respect to some monomial order. In fact, in this situation, the initial forms of the sequence are a regular sequence of monomials. 


\begin{proposition}\label{bdg}
Let ${\bf f}$ be an ideal of $A$ generated by a regular sequence $f_1,\ldots,f_r$ of degree $\ul{a}$, such  that $\{f_1,\ldots,f_r\}$ is a Gr\"obner basis with respect to some monomial order $\succcurlyeq$. Then, every homogeneous ideal of $A$ containing ${\bf f}$ satisfies $\EGHHH{\ul{a}}$.
\end{proposition}
\begin{proof}
Let $I$ be a homogeneous ideal that contains ${\bf f}$. Let us consider the set $\mathscr S$ of all homogeneous ideals of $A$ with the same Hilbert function as $I$ that contain a monomial regular sequence $g_1,\ldots,g_r$ of degree $\ul{a}$. Observe that $\mathscr S$ is not empty since, by assumption,  the initial ideal of $I$ contains the regular sequence of monomials given by the initial forms of $f_1,\ldots,f_r$, which has degree $\ul{a}$.

Since the  monomials $g_1,\ldots,g_r$ are pairwise coprime, we may write $g_i = \prod_{j \in B_i} x_j^{b_{ij}}$,\, for some subsets $B_i \subseteq \{1,\ldots,n\}$\, with\, $B_i \cap B_{i'} = \emptyset$\, if\, $i \ne i'$,\, and we let $|g_1,\ldots,g_r| =\sum_{i=1}^r |B_i|$ denote the cardinality of the support of  $g_1,\ldots,g_r$.

Now, we choose an element $J$ of $\mathscr S$ which contains a regular sequence $h_1,\ldots,h_r$ with minimal support and we will show that $|h_1,\ldots,h_r| = r$. In this way we will have that each $h_k$ is the $a_k$-th 
power of a variable, which
we may assume being equal to $x_k^{a_k}$; the conclusion will then follow by Theorem \ref{cleli}.

Clearly $|h_1,\ldots,h_r| \geq r$. If we assume by way of contradiction that the inequality were strict, then there would exist $i\in\{1,\ldots,r\}$ and $1\leq j < j'\leq n$ such that $x_jx_{j'} \,|\,h_i$. Consider then the change of coordinates $\varphi$ defined by
$$x_k \mapsto x_k,\,\,\,\text{ for all }\,k \ne j',\,\,\,\,\,\,\text{ and }x_{j'} \mapsto x_j+x_{j'},$$ let $J'=\ini_{\geq}(\varphi(J))$, where  $\geq$ denotes the lexicographic order, and let $h_k'=\ini_\geq(\varphi(h_k))\in J'$\, for $k=1,\ldots, r$. It is immediate to see that $h_1',\ldots,h_r'$ is still a monomial regular sequence of degree $\ul{a}$; since $J'$ has the same Hilbert function as $I$,\, it belongs to $\mathscr S$. However, $h_k'=h_k$ for all $k\ne i$, whereas $h_i'$ has one less variable than $h_i$ in its support. In particular, $|h_1',\ldots,h_r'| < |h_1,\ldots,h_r|$, which contradicts the minimality of the support of $h_1,\ldots,h_r$, and we are done. 
\end{proof}




Clearly, one can generalize the above by using a weight order $\omega$, as long as the given regular sequence form a Gr{\"o}bner basis with respect to the induced order $\geq_\omega$ and the ideal of the initial forms of the sequence satisfies the EGH Conjecture.

Contrary to the ``special'' case in which the regular sequence $f_1,\ldots,f_r$ is a Gr{\"o}bner basis, as far as we know the ``generic'' version of the conjecture is still open. We record this fact as a question.


\begin{question}\label{EGHgenerica}
Let $\un{a}= (a_1,\ldots,a_r)$ be a degree sequence. Does there exist a non-empty Zariski open set $U \subseteq \PP(A_{a_1})\times\PP(A_{ a_2})\times\cdots\times\PP(A_{a_r})$ of general forms of degree $\ul{a}$  such that, for every $[f_1,\ldots,f_r] \in U$, any ideal $I$ containing ${\bf f} = (f_1,\ldots,f_r)$ satisfies $\EGHHH{\ul{a}}$?
\end{question}


In \cite[Proposition 4.2]{HePo}, Herzog and Popescu show that, once a regular sequence of degree $\ul{a} = (2,2\ldots,2)$ is fixed, then any generic ideal generated by quadrics that contains it satisfies $\EGHHH{\ul{a}}$. We would like to warn the reader that Question \ref{EGHgenerica} addresses a different kind of ``genericity''. In fact, we are not fixing the regular sequence beforehand, but we are asking whether the EGH Conjecture holds for any ideal containing a general regular sequence.



\begin{remark}\label{Remark monomial basis}\begin{enumerate}[(1)] \item When  ${\bf f}$ is a general complete intersection, then the set of monomials of $A$ which do not belong to  the monomial complete intersection of the same degree as ${\bf f}$ 
forms a $K$-basis of $A/{\bf f}$, and this is well-known. This observation could be helpful in giving a positive answer to Question \ref{EGHgenerica}.
    \item It is currently not known, though, whether or not, after a general change of coordinates $\varphi:A \to A$ the set of monomials of $A$ which do not belong to  the monomial complete intersection of the same degree as ${\bf f}$ is a $K$-basis of $A/\varphi({\bf f})$, when ${\bf f}$ is a complete intersection. A positive answer in this matter  would make Question \ref{EGHgenerica} even more interesting. In fact, in light of the first part of the remark, it would provide a strategy to attack the EGH Conjecture at once.
  \end{enumerate}
\end{remark}
  \bigskip


There are some other very special cases when EGH is known to hold that can be found in the literature; we complete this section with two of them.


A special case of interest is when $I$ contains a regular sequence of quadrics, 
and this is the assumption on $I$ in the original statement of the conjecture. In this case, EGH is known to be true in low dimension; for $n \leq 4$, it can be proven by a direct application of linkage; see also \cite{Ch}. The validity of the conjecture for $n= 5$  was first claimed in \cite{Ri}, but a proof was never provided until recently, when  G{\"u}nt{\"u}rk{\"u}n and Hochster finally settle the case of five quadrics in  \cite[Theorem 4.1]{GuHo}.
We present an alternative proof of their result which relies on the techniques we used so far.



\begin{theorem} $I \subseteq A=K[x_1,\ldots,x_n]$  be a homogeneous ideal containing a regular sequence of degree $\ul{a} = (2,2,2,2,2)$; then, $I$ satisfies $\EGHHH{\ul{a}}$.
\end{theorem}

\begin{proof}
  We may assume that $K = \overline{K}$. By Proposition \ref{Prop Artinian reduction} we may assume that $n=5$,\,  $A/I$ is Artinian and ${\bf f} \subseteq I$ is an ideal generated by a regular sequence of five quadrics; notice that the socle degree $s$ of $A/{\bf f}$ is $s=5$.

  \noindent By Proposition \ref{Prop link} (ii) it suffices to show that $I$ satisfies $\EGHH{\ul{a}}{j}$ for $j=0,1,2$; this is clearly true for $j=0,1$ and we are left with the case $j=2$.

  If $H(I;2) = 6$, then we are done by Theorem \ref{Thm Franscisco}. Since the locus of reducible elements in $\PP({\rm Sym}^2(A_1))$ has dimension $2n-2 = 8$, if $H(I;2) \geq 7$ then $I$ must contain a reducible quadric $Q=\ell_1 \ell_2$. Proceeding as in the proof of Theorem \ref{Thm Abedelfatah}, we construct ideals $J_{[0]}$ and $J_{[1]}$ in $\tilde{A} = K[x_2,\ldots,x_5]$ such that $J = J_{[0]} \oplus J_{[1]} x_1 \oplus A x_1^2$ is a monomial vector space which contains $\aa = (x_1^2,\ldots,x_5^2)$,\, $\m_1 J_2 \subseteq J_3$, $H(A/J;i) = H(A/I;i)$ for all $i$, and the conclusion follows from an application of Theorem \ref{cleli}.
\end{proof}



In \cite{Co}, Cooper proves some cases of the EGH Conjecture when $r$ is small, including $\ul{a} = (a_1, a_2, a_3)$ with $a_1 = 2, 3$ and $a_2 = a_3$. We present a proof of the case $\ul{a} = (3, a, a)$, which is based on the techniques of \cite{CaDeS1}.


\begin{proposition} \label{Proposition Cooper}  Let $I \subseteq A=K[x_1,\ldots,x_n]$ be a homogeneous ideal containing a regular sequence of degree $\ul{a} = (3,a,a)$. Then,  $I$ satisfies $\EGHHH{\ul{a}}$.
\end{proposition}

\begin{proof}
  We may assume, as accustomed, that $K$ is infinite and, by Proposition \ref{Prop Artinian reduction}, that $r=n=3$. Therefore, let ${\bf f} =(f_1,f_2,f_3) \subseteq I$ be an ideal generated by a regular sequence of degree $\ul{a}$; since the socle degree of $A/I$ is $2a$, by Proposition \ref{Prop link} (ii), we only have to show that $I$ satisfies $\EGHH{\ul{a}}{d}$ for all $d < a$.

  Let $\ov{\ul{a}} = (3,a)$, and observe that $I$ satisfies $\EGHHH{\ov{\ul{a}}}$ by Theorem \ref{Thm EGH large degrees}. Thus, since $\EGHH{\ul{a}}{d}$ is equivalent to $\EGHH{\ov{\ul{a}}}{d}$ for all $d<a-1$, we only have to prove that $\EGHH{\ul{a}}{a-1}$ holds.

  \smallskip
  Let $Q=(f_1,f_2,u_1,\ldots,u_c)\subseteq I$, where $u_1,\ldots,u_c$ are the pre-images of a $K$-basis of $(I/{\bf f})_{a-1}$. First, assume that $f_3 \notin Q$; thus, $Q$ satisfies $\EGHH{\ov{\ul{a}}}{a-1}$ and, therefore, if $\ov{J}$ denotes  the smallest $\ov{\ul{a}}$-lpp ideal such that $H(Q;a-1) = H(\ov{J};a-1)$, we then have $H(Q;a) \geq H(\ov{J};a)$. Observe that $J = \ov{J}+(x_3^a)$ is an $\ul{a}$-lpp ideal such that $H(J;a-1) = H(\ov{J};a-1)$ and $H(J;a) = H(\ov{J};a)+1$. We then have that $$H(I;a) \geq H(Q+(f_3);a) = H(Q;a)+1 \geq H(\ov{J};a)+1 = H(J;a),$$ and this case is  done.

  \noindent
  Otherwise, $f_3 \in Q$ and, accordingly,  $\Ht(Q)=3$. By Prime Avoidance we may assume that $f_1,v_c,f_2$ forms a regular sequence of degree $\ul{a}' = (3,a-1,a)$ when $a \ne 3$;  when $a=3$, we may take the sequence $v_c,f_1,f_2$ of degree  $\ul{a}'=(2,3,3)$ instead. Either way, $I$ satisfies $\EGHHH{\ul{a}'}$ by Theorem \ref{Thm EGH large degrees} and, therefore, there exists a $\ul{a}'$-lpp ideal $J$ with the same Hilbert function as $I$. In particular, since $\ul{a} \geq \ul{a}'$, the monomial ideal $J$ also contains $\aa=(x_1^3,x_2^a,x_3^a)$, and we conclude by Theorem \ref{cleli}.
\end{proof}

\section{Applications and examples} \label{Section Applications Examples}
In this section, we present some applications of the EGH Conjecture, supported by several examples. For our computations, it is convenient to introduce the following integers.


\begin{definition}\label{binomiale}
  Let $\un{a} = (a_1,\ldots,a_r)$ be a degree sequence, and $h,d$ be non-negative integers with $h \leq n$ and $d \geq 1$. For $r<i\leq n$,\, we let\, $a_i = \infty$\, and\, $x_i^{a_i} = 0$. Also,  we let
\[ \begin{bmatrix} h\\ d
    \end{bmatrix}_{\un{a}} = \begin{cases} \displaystyle \dim_K \left(\frac{K[x_{n-h+1},\ldots,x_n]}{(x_i^{a_i} \,\mid \, n-h+1 \leq i \leq n)}\right)_{\hspace{-1mm}\text{\normalsize $d$}} & \text{ if } h \geq 1; \\ \\ 0 & \text{ if } h=0. \end{cases}
    \]
\end{definition}

Whenever $\un{a}$ is clear from the context, we will omit it from the notation. 


\begin{remark} Notice that $\begin{bmatrix} h\\ d \end{bmatrix}_{\ul{a}}$ actually depends on $n$: for instance $\begin{bmatrix}
    1 \\ 2 \end{bmatrix}_{(2)}=\begin{cases} 0 \text{ if } n=1;\\
  1\text{ otherwise. }
  \end{cases}$
\end{remark}


 \noindent
 The next definition is based on  the Macaulay representation, cf.  \cite[Section 4.2]{BrHe}, but it takes also into account the additional information brought by the degree sequence.
 


We adopt the standard convention that $\infty-1 = \infty$ and $a \leq \infty$ for all $a \in \ZZ$.

With the above notation, given an integer $0< k \leq \begin{bmatrix} n \\ d \end{bmatrix}$, we may write
\[
k = \begin{bmatrix}
k_d \\ d
\end{bmatrix} + \begin{bmatrix}
k_{d-1} \\ d-1 
\end{bmatrix} + \ldots + \begin{bmatrix}
k_1 \\ 1
\end{bmatrix},
\]
where $k_d \geq k_{d-1} \geq \ldots \geq k_1 \geq 0$ and $\#\{t \mid k_t = i\} \leq \begin{cases} a_{n-i}-1 & \text{ for } 0 \leq i < n; \\  1 & \text{ for } i=n. \end{cases}$

\noindent Such an expression is called the {\it $(\ul{a},n)$-Macaulay representation of $k$ in base $d$}. As for the classical Macaulay representation, which corresponds to the choice $a_i=\infty$ for all $i$, the $(\ul{a},n)$-Macaulay representation of $k$ in base $d$ exists, and it is unique; for instance, see \cite{GrKl,RiSa,Co2}.

Finally, given the $\ul{a}$-Macaulay representation of $k$ in base $d$, we let

\[
k^{\langle d\rangle}_{\un{a}}=\begin{bmatrix} k_d\\ d+1
    \end{bmatrix} + \begin{bmatrix} k_{d-1}\\ d
    \end{bmatrix}+\ldots+\begin{bmatrix} k_1\\ 2
    \end{bmatrix}.
\]

\noindent Observe that, given any $\ul{a}$-lpp ideal $J \subseteq A$ with $k=H(A/J;d)$, then  $H(A/(\m J + \aa);d+1) = k^{\langle d\rangle}_{\ul{a}}$. 

\medskip

\noindent Since $\m_1 J_d \subseteq J_{d+1}$, the $k^{\langle d\rangle}_{\ul{a}}$ represents the maximal growth in degree $d+1$ of the quotient by an $\ul{a}$-lpp ideal which has Hilbert function equal to $k$ in degree $d$, as it happens in the classical case.

\noindent Next, we present a proof of the following enhanced version of Macaulay Theorem, see for instance \cite{RiSa,Co2}, which is  a direct consequence of Theorem \ref{cleli}.



\begin{theorem}\label{Thm Macaulay}
  Let $\ul{a} = (a_1,\ldots,a_r)$ be a degree sequence, $\aa$  the corresponding pure-powers ideal, and  $R=A/\aa$. Let $H \,:\, \NN \lra \NN$ be a numerical function; 
  then, $H$ is the Hilbert function of $R/I$ for some homogeneous ideal $I$ of $R$  if and only if $$H(d+1)\leq H(d)^{\langle d \rangle}_{\un{a}}, \text { for all } d \geq 1.$$
\end{theorem}

\begin{proof}
Let $I \subseteq R$ be a homogeneous ideal and $J$ its lift to $A$. By Theorem \ref{cleli},  $L=\LPP{\ul{a}}{J}$ is an ideal with the same Hilbert function as $I$, and from the fact that $\m_1 J_d \subseteq J_{d+1}$ we get that $H(A/J;d+1)=H(A/L;d+1)\leq H(A/L;d)^{\langle d\rangle}_{\ul{a}}= H(A/J;d)^{\langle d\rangle}_{\ul{a}}$.

\smallskip
Conversely, let $H$ be a numerical function that satisfies the growth condition, $d$ be a non-negative integer, and let $V \subseteq A_d$ be an $\ul{a}$-lpp $K$-vector space such that $\dim_K(A_d/V) = H(d)$. Consider the $\ul{a}$-lpp ideal $J = (V) + \aa$; then  $H(d)^{\langle d\rangle }_{\ul{a}}$ coincides with the dimension of $(A/J)_{d+1}$ which, by assumption, is at least $H(d+1)$. By adding appropriate monomials to $J_{d+1}$ if necessary, we can make  $J$ into  an $\ul{a}$-lpp $K$-ideal such that $\dim_K((A/J)_{d+1}) = H(d+1)$. Arguing in this way for all $d$, we obtain a monomial ideal $I$ containing $\aa$, which in fact is an $\ul{a}$-lpp ideal, with Hilbert function $H$.
\end{proof}


\noindent
There are implementations of these results in software systems such as Macaulay2, see for instance the one authored by White \cite{Wh}.


\begin{example} Let $A=K[x_1,x_2,x_3]$, and let $I\subseteq A$ be a homogeneous ideal which contains a regular sequence of degree $\ul{a} = (3,3,4)$. Suppose that, regarding its Hilbert function, we only know that $H(A/I;5) = 5$, and that we would like to estimate $H(A/I;6)$. Classically, this is achieved by means of Macaulay Theorem, which provides  $H(A/I;6) \leq 5$. However, since $\EGHHH{\ul{a}}$ holds by Theorem \ref{Thm EGH large degrees}, we know that $H(I) = H(\LPP{\ul{a}}{I})$, therefore Theorem \ref{Thm Macaulay} yields that $H(A/I;6) \leq 5^{\langle 5 \rangle}_{\ul{a}} = 2$. 
\end{example}

The following result was observed by Liang \cite{Liang}.


\begin{proposition} \label{Prop mingens}
Let $I \subseteq A = K[x_1,x_2,x_3]$ be an ideal which contains an ideal ${\bf f}$ generated by a regular sequence of degree $(a_1,a_2)$ and let $\mu(I)$ denote its  minimal number of generators; then, $\mu(I)\leq a_1\cdot a_2$. 
\end{proposition}

\begin{proof} Observe that any ideal containing ${\bf f}$ satisfies $\EGHHH{(a_1,a_2)}$ by Theorem \ref{Thm EGH large degrees}, therefore by Lemma \ref{Lemma lpp growth} we have $H(I/\m I) \leq H(L/\m L)$, where $L=\LPP{(a_1,a_2)}{I}$. Thus, we may as well bound $\mu(L)$. Notice that, if $u = x_1^i x_2^jx_3^k$ is a minimal generator of $J$, then $0 \leq i < a_1$ and $0 \leq j < a_2$, since $J$ contains $\aa=(x_1^{a_1},x_2^{a_2})$. Moreover, if $v = x_1^{i'}x_2^{j'}x_3^{k'}$ is another minimal monomial generator of $J$, then necessarily $i' \ne i$ or $j' \ne j$. Therefore, there are at most $a_1 \cdot a_2$ possible choices for $i$ and $j$, as desired.
\end{proof}



\noindent
Proposition \ref{Prop mingens} can be applied to bound the number of defining equations of curves in $\PP^3$. In fact, such a curve can is defined by a homogeneous height two  ideal $P \subseteq K[x_0,x_1,x_2,x_3]$, which then contains a regular sequence of some degree $(a_1,a_2)$. Pick a general linear form $\ell$ which is regular modulo $P$ and let $\ov{A} = A/\ell \cong K[x_1,x_2,x_3]$,\, and $\ov{P}=P\ov{A}$. Then  $\mu(\ov{P})=\mu(P)$,  and use Proposition \ref{Prop mingens} on $\ov{P}$, since the latter contains a regular sequence of degree $(a_1,a_2)$.

As we mentioned in the introduction, see Conjecture \ref{Conj CB}, another application of the EGH Conjecture is the Cayley-Bacharach Theorem. Its original formulation  states that a cubic $\mathscr C \subseteq \PP^2$ which contains eight points that lie on the intersection of two cubics, must contain the ninth point as well. Later on, this fact  has been extended and generalized in various ways. We illustrate a connection with the EGH  in the following example.

\begin{example} \label{Ex CB}
  Let $X \subseteq \PP^3$ be a complete intersection of degree $(3,3,3)$. We show that a cubic hypersurface $Y$ containing at least $22$ of the $27$ points of $X$, must contain $X$.

  To see this, let ${\bf f} = (f_1,f_2,f_3) \subseteq A=K[x_1,\ldots,x_4]$ be an ideal of definition of $X$. Moreover,   let $g$ be a cubic defining $Y$, let $I = {\bf f}+(g)$ and, by way of contradiction, assume that $g \notin {\bf f}$. Let $|K|=\infty$; after a general change of coordinates, if necessary, we may write $I^{{\rm sat}} = (I:_A x_4^\infty)$ and assume that $x_4$ is $A/I^{{\rm sat}}$-regular. 

     Clearly, $g\in I^{\rm sat}$. Next, we claim that we may assume that $g \notin {\bf f} + (x_4)$. In fact, if  this is not the case, there exists $0\neq g_1 \in (I:_Ax_4) \subseteq I^{\rm sat}$ of degree at most $2$ such that $g=f+g_1x_4$, for some $f\in {\bf f}$.  The element $g_1$ may or may not belong to ${\bf f}+(x_4)$. If it does, arguing as above, we  obtain that $I^{\rm sat}$ actually contains a linear form $\ell$, which is not in ${\bf f}+(x_4)$, since $x_4$ is $A/I^{\rm sat}$-regular. Either way,  we found an element $g_2\in I^{\rm sat}$ of degree $<3$  which does not belongs to  ${\bf f}+(x_4)$. Multiplying it by an appropriate power of $x_4$, we obtain a form $g_3$ of degree $3$ which still belongs to $I^{\rm sat}$, but does not belong to ${\bf f}+(x_4)$. Therefore, we may let $g=g_3$, and our claim is proven.

     \smallskip
     Henceforth, let $\ov{A}=A/(x_4)$ and denote by  ${\bf \ov{f}},\,\, \ov{I},\,\, \ov{g}$\, and \, $\ov{I^{\rm sat}}$ the images in $\ov{A}$\,  of\, ${\bf f},\, I,\, g$\, and \,$I^{\rm sat}$ respectively; moreover, let $J = {\bf \ov{f}} + (\ov{g}) \subseteq \ov{A}$. Then, we immediately have
     $${\rm e}(A/I) = {\rm e}(A/I^{\rm sat})={\rm e}(\ov{A}/\ov{I^{\rm sat}}) \leq {\rm e}(\ov{A}/J).$$ 
By Proposition \ref{Proposition Cooper}, $\LPP{(3,3,3)}{J}$ is an ideal with the same Hilbert function as $J$. Moreover, since $\ov{g} \notin {\bf \ov{f}}$ by what we have seen above, the ideal $\LPP{(3,3,3)}{J}$ must contain the monomial $x_1^2x_2$. In particular, $${\rm e}(A/I)\leq {\rm e}(\ov{A}/J) =  {\rm e}(\ov{A}/\LPP{(3,3,3)}{J}) \leq {\rm e}(\ov{A}/(x_1^3,x_1^2x_2,x_2^3,x_3^3)) = 21.$$ However, our hypothesis guarantees that ${\rm e}(A/I) \geq 22$, a contradiction. 
\end{example}

We conclude the paper by illustrating how the combinatorial Kruskal-Katona Theorem \cite{Kr,Ka}, a characterization of all the possible $f$-vectors of simplicial complexes $\Delta$, is related to the EGH Conjecture for $\ul{a} = (2,2,\ldots,2)$. 
For additional details on what follows, see for instance \cite[Section 6.4]{HeHi}. 



\noindent
Recall that the $f$-vector $f(\Delta)=(f_0,\ldots,f_{r-1})$ of an $(r-1)$-dimensional simplicial complex $\Delta$ simply records in its entry $f_{i-1}$ the number of faces of $\Delta$ of dimension $i-1$. As it is customary, we set $f_{-1}=1$. Given positive integers $h,d$, write its Macaulay representation 
$
h = {h_d \choose d} + {h_{d-1} \choose d-1} + \ldots + {h_1 \choose 1},
$
where $h_d \geq  h_{d-1} \geq \ldots \geq h_1 \geq 0$, and  set 
$$
h^{(d)} = {h_d \choose d+1} + {h_{d-1} \choose d} + \ldots + {h_1 \choose 2};
$$
the Kruskal-Katona Theorem states that $(f_0,\ldots,f_{r-1})$ is the $f$-vector of a simplicial complex of dimension $r-1$ if and only if $f_{d+1} \leq f_d^{(d+1)}$ for each $d=0,\ldots,r-2$. 


\smallskip
\noindent
Given a simplicial complex $\Delta$, its $f$-vector $f(\Delta)$ and its Stanley-Reisner ring $K[\Delta]$, we have that  $K[\Delta] = K[x_1,\ldots,x_n]/J$, where $n = f_0$ and $J=J_\Delta$ is a square-free monomial ideal. If we let $R=K[x_1,\ldots,x_n]/I$, where $I =J+(x_1^2,\ldots,x_n^2)$, then it is easy to see that $H(R;i) = f_{i-1}$ for all $i \geq 0$.

On the other hand, any monomial ideal $I \subseteq A=K[x_1,\ldots,x_n]$ containing $\aa= (x_1^2,\ldots,x_n^2)$, can be written uniquely as $I=J+\aa$, where $J$ is a square-free monomial ideal. If we consider $\Delta=\Delta_J$, then its $f$-vector $f(\Delta)=(f_0,\ldots,f_{r-1})$, where $f_i = H(A/I;i+1)$ for all $i\geq 0$.

\smallskip
\noindent
Finally, the crucial  observation is that  $\begin{bmatrix} k \\ d \end{bmatrix}_{\ul{a}} = \displaystyle {k \choose d}$ when if $\ul{a}=(2,2,\ldots,2)$.

\noindent
Therefore, the numerical condition of Theorem \ref{Thm Macaulay} can be restated as
$$f_d = H(R;d+1) \leq H(R;d)^{\langle d \rangle}_{\ul{a}} = H(R;d)^{(d)} = f_{d-1}^{(d)},\,\,\,\text{for all }\, d \geq 1,$$  which is precisely the condition of  Kruskal-Katona Theorem. 

\begin{example} 
Let $f=(4,5,2)$, and let us construct a simplicial complex $\Delta$ such that $f(\Delta)=f$. Consider the numerical function $H:\NN \to \NN$ defined as $H(0)=1$, $H(1)=4$, $H(2) = 5$, $H(3)=2$, and $H(d) = 0$ for $d>3$. By means of Theorem \ref{Thm Macaulay}, it can be checked that there exists a $(2,2,2,2)$-lpp ideal $I$ with Hilbert function equal to $H$, namely, $I=(x_1x_2) + (x_1^2,x_2^2,x_3^2,x_4^2)$. If we let $J=(x_1x_2)$, then $\Delta=\Delta_J$ is the following $2$-dimensional simplicial complex

\begin{center}
\begin{tikzpicture}
\draw (0,0) node[anchor=north]{$x_3$}
  -- (4,0) node[anchor=north]{$x_1$}
  -- (5.5,3.5) node[anchor=south]{$x_4$}
  -- cycle;
\path[pattern=vertical lines,pattern color=green] (0,0)--(4,0)--(5.5,3.5)--cycle;
\draw (0,0) node[anchor=north]{}
  -- (1.6,3.5) node[anchor=east]{$x_2$}
  -- (5.5,3.5) node[anchor=south]{}
  -- cycle;
\path[pattern=horizontal lines,pattern color=red] (0,0)--(1.6,3.5)--(5.5,3.5)--cycle;
\end{tikzpicture}
\end{center}
and  $f(\Delta)=f$.
\end{example}



\begin{example}
If $f=(4,5,3)$, then there is no simplicial complex $\Delta$ such that $f(\Delta)=f$, since  there is no $(2,2,2,2)$-lpp ideal of $K[x_1,x_2,x_3,x_4]$ with Hilbert function $H$ satisfying $H(2)=5$ and $H(3)=3> H(2)^{\langle 2\rangle}_{(2,2,2,2)} = 2$.
\end{example} 


\end{document}